\documentclass{amsart}

\usepackage{amsmath}
\usepackage{amsfonts}
\usepackage{amssymb}
\usepackage{amsthm}
\usepackage{comment}
\usepackage{epsfig}
\usepackage{psfrag}
\usepackage{mathrsfs}
\usepackage{amscd}
\usepackage[all]{xy}
\usepackage{rotating}
\usepackage{lscape}
\usepackage{amsbsy}
\usepackage{verbatim}
\usepackage{moreverb}
\usepackage{color}
\usepackage{bbm}
\usepackage{eucal}

\newtheorem{theorem}{Theorem}[section]
\newtheorem{prop}[theorem]{Proposition}
\newtheorem{lemma}[theorem]{Lemma}
\newtheorem{cor}[theorem]{Corollary}

\theoremstyle{definition}
\newtheorem{definition}[theorem]{Definition}

\newtheorem{ack}[theorem]{Acknowledgments}
\newtheorem{observation}[theorem]{Observation}

\newtheorem{construction}[theorem]{Construction}

\theoremstyle{remark}
\newtheorem{remark}[theorem]{Remark}

\newtheorem{example}[theorem]{Example}

\DeclareMathOperator{\Psh}{\sf PShv}

\DeclareMathOperator*{\colim}{\sf colim}
\DeclareMathOperator*{\limit}{\sf lim}

\DeclareMathOperator{\Fun}{\sf Fun}

\DeclareMathOperator{\Map}{\sf Map}
\DeclareMathOperator{\Mapc}{{\sf Map}_{\sf c}}
\DeclareMathOperator{\Gammac}{{\Gamma}_{\!\sf c}}

\DeclareMathOperator{\Space}{\sf Spaces}

\DeclareMathOperator{\sym}{\sf Sym}

\DeclareMathOperator{\Cat}{\sf Cat_\infty}
\DeclareMathOperator{\uno}{\mathbbm{1}}

\DeclareMathOperator{\m}{\sf Mod}

\DeclareMathOperator{\Alg}{\mathsf{Alg}}

\DeclareMathOperator{\op}{\mathsf{op}}
\DeclareMathOperator{\com}{\mathsf{Com}}

\DeclareMathOperator{\Top}{\mathsf{Top}}
\DeclareMathOperator{\Emb}{\mathsf{Emb}}

\DeclareMathOperator{\conf}{\mathsf{Conf}}

\DeclareMathOperator{\spaces}{\mathsf{Spaces}}

\DeclareMathOperator{\mfld}{{\cM}\mathsf{fld}}
\DeclareMathOperator{\dmfld}{\mathsf{Mfld}}
\DeclareMathOperator{\sm}{\mathsf{sm}}
\DeclareMathOperator{\fr}{\sf fr}

\DeclareMathOperator{\Sing}{\mathsf{Sing}}

\DeclareMathOperator{\BTop}{\sf BTop}

\DeclareMathOperator{\BO}{{\mathsf BO}}

\DeclareMathOperator{\Lie}{\sf Lie}

\def\ot{\otimes}

\DeclareMathOperator{\fin}{\sf Fin}

\DeclareMathOperator{\oo}{\infty}

\DeclareMathOperator{\hh}{\sf HC}

\DeclareMathOperator{\free}{\sf Free}

\DeclareMathOperator{\disk}{{\cD}\sf isk}
\DeclareMathOperator{\ddisk}{{\sf Disk}}

\DeclareMathOperator{\bdelta}{\boldsymbol{\Delta}}

\newcommand{\ra}{\rightarrow}

\newcommand{\xra}{\xrightarrow}

\def\cC{\mathcal C}\def\cD{\mathcal D}
\def\cE{\mathcal E}\def\cF{\mathcal F}
\def\cI{\mathcal I}\def\cJ{\mathcal J}\def\cK{\mathcal K}
\def\cM{\mathcal M}\def\cO{\mathcal O}
\def\cS{\mathcal S}
\def\cV{\mathcal V}\def\cX{\mathcal X}

\def\AA{\mathbb A}\def\DD{\mathbb D}
\def\EE{\mathbb E}

\def\RR{\mathbb R}

\def\ZZ{\mathbb Z}

\def\sB{\mathsf B}\def\sC{\mathsf C}\def\sD{\mathsf D}
\def\sH{\mathsf H}

\def\sN{\mathsf N}\def\sO{\mathsf O}

\def\sU{\mathsf U}

\def\bH{\mathbf H}

\begin{document}

\title{Factorization homology of topological manifolds}
\author{David Ayala \& John Francis}
\address{Department of Mathematics\\Montana State University\\Bozeman, MT 59717}
\email{david.ayala@montana.edu}
\address{Department of Mathematics\\Northwestern University\\Evanston, IL 60208}
\email{jnkf@northwestern.edu}
\thanks{DA was partially supported by ERC adv.grant 228082 and by the National Science Foundation under Award 0902639. JF was supported by the National Science Foundation under award 0902974 and 1207758; part of the writing was completed while JF was a visitor at Universit\'e Pierre et Marie Curie in Jussieu as a guest of the Foundation Sciences Math\'ematiques de Paris.}

\begin{abstract} Factorization homology theories of topological manifolds, after Beilinson, Drinfeld and Lurie, are homology-type theories for topological $n$-manifolds whose coefficient systems are $n$-disk algebras or $n$-disk stacks. In this work we prove a precise formulation of this idea, giving an axiomatic characterization of factorization homology with coefficients in $n$-disk algebras in terms of a generalization of the Eilenberg--Steenrod axioms for singular homology. Each such theory gives rise to a kind of topological quantum field theory, for which observables can be defined on general $n$-manifolds and not only closed $n$-manifolds. For $n$-disk algebra coefficients, these field theories are characterized by the condition that global observables are determined by local observables in a strong sense. Our axiomatic point of view has a number of applications. In particular, we give a concise proof of the nonabelian Poincar\'e duality of Salvatore, Segal, and Lurie. We present some essential classes of calculations of factorization homology, such as for free $n$-disk algebras and enveloping algebras of Lie algebras, several of which have a conceptual meaning in terms of Koszul duality.
\end{abstract}



\maketitle

\tableofcontents

\section{Introduction}

Factorization homology takes an algebraic input, either an $n$-disk algebra or more generally a stack over $n$-disk algebras, and outputs a homology-type theory for $n$-dimensional manifolds. In the case where the input coefficients are an $n$-disk algebra, this factorization homology -- the topological chiral homology introduced by Lurie \cite{dag} -- satisfies an analogue of the axioms of Eilenberg and Steenrod. This work proves this statement and some consequences afforded by this point of view.

\smallskip

While the subject of factorization homology is new, at least in name, it has important roots and antecedents. Firstly, it derives from the factorization algebras of Beilinson and Drinfeld \cite{bd}, a profound and elegant algebro-geometric elaboration on the role of configuration space integrals in conformal field theory. Our work is in essence a topological version of theirs, although the topological setting allows for arguments and conclusions ostensibly unavailable in the algebraic geometry.  Secondly, it has an antecedent in the labeled configuration space models of mapping spaces dating to the 1970s; it is closest to the models of Salvatore \cite{salvatore} and Segal \cite{segallocal}, but see also \cite{kallel}, \cite{bodig}, \cite{mcduff}, \cite{may}, and \cite{segal}. Factorization homology thus lies at the broad nexus of Segal's ideas on conformal field theory \cite{segalconformal} and his ideas on mapping spaces articulated in \cite{segal} and \cite{segallocal}.

\smallskip

In keeping with these two directions, we offer two primary motivations for the study of factorization homology. Returning to our first point, factorization homology with coefficients in $n$-disk algebras are homology theories for topological manifolds satisfying a generalization of the Eilenberg--Steenrod axioms for ordinary homology; as such, it generalizes ordinary homology in a way that is only defined on $n$-manifolds and not necessarily on arbitrary topological spaces.
Second, these homology theories define topological quantum field theories. Following the vision of Costello--Gwilliam \cite{kevinowen}, factorization homology with coefficients in $n$-disk stacks offers an a algebraic model for the observables in a general topological quantum field theory. The special case of $n$-disk algebra coefficients corresponds to $n$-dimensional field theories whose global observables are determined by the local observables, as in a perturbative quantum field theory.

\smallskip

We first elaborate on the homology theory motivation, which serves as the main artery running through this work. One can pose the following question: what can a homology theory for topological manifolds be? Singular homology, of course, provides one answer. One might not want it to be the only answer, since singular homology is not specific to manifolds and can be equally well defined for all topological spaces; likewise, it is functorial with respect to all maps of spaces, not just those that are specifically meaningful for manifolds (such as embeddings or submersions). One could thus ask for a homology theory which is specific to manifolds, not defined on all spaces, and which might thereby distinguish manifolds which are homotopic but not homeomorphic and distinguish embeddings that are homotopic but not isotopic through embeddings. That is, one could ask for a homology theory for manifolds that can detect the more refined and interesting aspects of manifold topology. The question then becomes, do such homology theories exist?

\smallskip

One might address this question by first making it more precise, by defining exactly what one means by a homology theory for manifolds; since we know what a homology theory for spaces constitutes, by the Eilenberg--Steenrod axioms, one might simply modify those axioms as little as possible, but so as to work only for manifolds and with the maximum possibility that new theories might arise.

\smallskip

Reformulating the Eilenberg--Steenrod axioms slightly, one can think of an ordinary homology theory $\cF$ as a functor $\Space^{\sf fin}\ra {\sf Ch}$ from the topological category of spaces homotopy equivalent to finite CW complexes to the topological category of projective chain complexes\footnote{There is the standard functor $|\Map_{\sf Ch}\bigl(\sC_\ast (\Delta^\bullet), - \bigr)|\colon {\sf Ch}\to \Top$ to topological spaces given by the Dold--Kan correspondence.
While this functor does not send tensor product to Cartesian product, it does so up to a coherent natural transformation.  
The internal hom-objects of ${\sf Ch}$ thus give a topological enrichment.} 
satisfying two conditions:

\begin{itemize}
\item The canonical morphism $\bigoplus_J \cF(X_\alpha) \ra \cF(\coprod_J X_\alpha)$ is an equivalence for any finite set $J$;
\item Excision: for any diagram of cofibrations of spaces $X' \hookleftarrow X \hookrightarrow X''$, the resulting map of chain complexes
\[\cF(X')\underset{\cF(X)} \oplus\cF(X'') \longrightarrow \cF\Bigl(X'\underset{X}\sqcup X''\Bigr)\] is a quasi-isomorphism.
\end{itemize}

The first condition can be restated as saying that the functor $\cF$ is symmetric monoidal with respect to the disjoint union and direct sum. Let $\bH(\spaces, {\sf Ch}^\oplus)$ stand for the collection of such functors. The Eilenberg--Steenrod axioms for ordinary homology can then be reformulated as follows.

\begin{theorem}[Eilenberg--Steenrod] Evaluation on a point, ${\sf ev}_*:\bH(\spaces, {\sf Ch}^\oplus)\ra {\sf Ch}$, defines an equivalence between homology theories valued in chain complexes with direct sum and chain complexes. The inverse is given by singular homology, the functor assigning to a chain complex $V$ the functor $\sC_*(-,V)$ of singular chains with coefficients in $V$.
\end{theorem}

In particular, each homology theory $\cF$ is equivalent to the functor of singular chains with coefficients in the chain complex $\cF(\ast)$, which is the value of $\cF$ on the singleton: 
$
\sC_\ast(-, \cF(\ast)) \simeq\cF.
$
Further, every natural transformation of homology theories $\cF \ra \cF'$ is determined by the map $\cF(\ast)\ra \cF'(\ast)$.

To adapt this definition to manifolds, we make two substitutions:
\begin{enumerate}
\item We replace $\spaces^{\sf fin}$ with $\mfld_n$, the collection of topological $n$-manifolds, not necessarily closed but with suitably finite covers, with embeddings as morphisms.
\item We replace the target ${\sf Ch}^\oplus$ by a general symmetric monoidal $\oo$-category $\cV$.
\end{enumerate}
As so, we define $\bH(\mfld_n,\cV)$ to be the collection of all symmetric monoidal functors from $\mfld_n$ to $\cV$ that satisfy a monoidal version of excision. We arrive at the following analogue of the Eilenberg--Steenrod axioms, for $\cV$ a symmetric monoidal $\oo$-category satisfying a technical condition (Definition \ref{ast}).

\begin{theorem} There is an equivalence between homology theories for topological $n$-manifolds valued in $\cV$ and $n$-disk algebras in $\cV$
\[\xymatrix{
\displaystyle\int: \Alg_{\disk_n}(\cV)\ar@<-.5ex>[r] &\bH(\mfld_n, \cV):{\sf ev}_{\RR^n}~.\ar@<-.5ex>[l]\\}\]
This equivalence is implemented by the factorization homology functor $\int$ from the left, and evaluation on $\RR^n$ from the right.
\end{theorem}

The $n$-disk algebras appearing in the theorem are equivalent to the $\cE_n$-algebras of Boardman and Vogt \cite{bv} together with an extra compatible action of the group of automorphisms of $\RR^n$. Thus, at first glance, this result appears different and more complicated than the Eilenberg--Steenrod axioms because the characterization as $\cV$ alone has been replaced by $n$-disk algebras in $\cV$. 
This is for two reasons, each of substance.
For one, the object $\RR^n$ has more structure than a singleton; for instance, automorphisms of $\RR^n$ form an interesting and noncontractible space, whereas the automorphisms of a singleton is just a point. 
Secondly, the symmetric monoidal structure of $\cV$ is not necessarily the coproduct, so there need not be a canonical map $V\ot V \to V$ for each object $V\in \cV$; this allows for a multiplicative form of homology.
Nonetheless, our result does specialize to Eilenberg and Steenrod's, as we shall see in Example~\ref{ES}. In particular, in the example $\cV = {\sf Ch}^\oplus$, these $n$-disk algebras are equivalent to just chain complexes with an action of the automorphisms of $\RR^n$, and these homology theories are then just ordinary homology twisted by the tangent bundle.

\smallskip

This result has a number of immediate applications and leads to new proofs of known results. For instance, it gives a one line proof that factorization homology of the circle, in the case $\cV$ is chain complexes with tensor product, is Hochschild homology. It also gives a new and short proof of the nonabelian Poincar\'e duality of Salvatore \cite{salvatore}, Segal~\cite{segallocal}, and Lurie \cite{dag}. We give further results and computations in Section 5. 

\smallskip

The second motivation for factorization homology comes from mathematical physics. In the various axiomatics for topological quantum field theory after Segal \cite{segalconformal}, one restricts to compact manifolds, possibly with boundary; the locality of a quantum field theory is then reflected in the functoriality of gluing cobordisms. However, it is frequently possible to define a quantum field theory on noncompact manifolds, such as Euclidean space. Since there are more embeddings between general noncompact manifolds than between closed manifolds, one might then expect the axiomatics of this situation to be slightly different were one to account for the structure in which one can restrict a field, or extend an observable, along an open embedding $M\hookrightarrow M'$.
\smallskip

Our notion of a homology theory for manifolds is thus simultaneously an attempt to axiomatize the structure of the observables in a quantum field theory which is topologically invariant -- here the $\ot$-excision of the homology theory becomes a version of the locality of the field theory -- and Theorem \ref{homology} becomes an algebraic characterization of part of the structure of such quantum field theories. There is another characterization of extended topological field theories, namely the Baez--Dolan cobordism hypothesis, Lurie's proof of which is outlined in \cite{cobordism}, building on earlier work with Hopkins and inspired by ideas of Costello \cite{costello2}. These two characterizations are comparable in several ways. In particular, there is a commutative diagram
\[\xymatrix{
\Alg_{\disk_n^{\sf sm}}(\cV)^{\sim}\ar[d]_{\int}\ar[r]&\bigl(\Alg_{n}(\cV)^{\sim}\bigr)^{\sO(n)}\ar[d]^Z\\
\bH(\mfld^{\sm}_n, \cV)^{\sim} \ar[r]& \Fun^\ot({\sf Bord}^{\sm}_n, \Alg_{n}(\cV))~.\\
}\]
\noindent
This picture should be understood only as impressionistic, and we briefly explain the terms in this picture (see~\S4.1 of~\cite{cobordism} for more explanation): $\disk_n^{\sm}$ and $\mfld^{\sm}_n$ are the $\oo$-categories of smooth $n$-disks and $n$-manifolds with smooth embeddings; ${\sf Bord}^{\sm}_n$ is the $(\oo,n)$-category of smooth bordisms of manifolds from \cite{cobordism}; and $\Alg_{n}(\cV)$ is the higher Morita category, where $k$-morphisms are $\disk_{n-k}^{\fr}$-algebras in bimodules; the superscript $(-)^{\sim}$ denotes underlying $\oo$-groupoids, discarding non-invertible morphisms. The bottom horizontal functor from homology theories valued in $\cV$ to topological quantum field theories valued in $\Alg_{n}(\cV)$, assigns to a homology theory $\cF$ the functor on the bordism category sending a $k$-manifold $M$ with corners to $\cF({M}^\circ\times\RR^{n-k})$, the value of $\cF$ on a collar-thickening of $M$.

\smallskip

Our characterization of homology theories can thereby be seen as an analogue of the cobordism hypothesis for the observables in a topological quantum field theory where one allows for noncompact manifolds and a strong locality principle by which local observables determine global observables.\footnote{There is a slight, but interesting, difference in that the homology theory characterization applies successfully to topological manifolds (as well as piecewise linear or smooth), whereas the cobordism hypothesis requires that the manifolds involved have at least a piecewise linear structure; the absence of triangulations for nonsmoothable topological 4-manifolds appears as a genuine obstruction.} Not all topological field theories come from such homology theories, and the question of which do and do not is an interesting one. Costello and Gwilliam use closely related ideas in studying more general quantum field theories which are not topologically invariant, and the setting of their work suggests that perturbative quantum field theories are exactly those amenable to this characterization; see \cite{kevinowen} and \cite{owen}. The structure of observables of a topological quantum field theory which is not perturbative is better described by a generalization of factorization homology with coefficients given by a stack over $n$-disk algebras \cite{thez}, e.g., an algebraic variety $X$ whose ring of functions $\cO_X$ is enhanced to have a compatible structure of an $n$-disk algebra.

\smallskip

Factorization homology and related ideas have recently become the subject of closer study; in addition to Lurie's originating work, see \cite{dag} and \cite{cobordism}, and Costello and Gwiliam \cite{kevinowen}, see also \cite{andrade}, \cite{gtz2}, \cite{owen}, and \cite{blob}. We expect the theory of factorization homology to be a source of interesting future mathematics and to carry many important manifold invariants. Especially fertile ground lies in low-dimensional topology, in the study of 3-manifold and knot invariants, where invariants stemming from the homology of configuration spaces are already prevalent. This is a source of work joint with Tanaka in \cite{aft2}, where we construct factorization knot homology theories from a $3$-disk algebra with extra structure.

\smallskip

It is a compelling general question as to how much of manifold topology can be captured by factorization homology; this question is closely related to the Goodwillie--Weiss manifold calculus \cite{weiss}. Factorization homology of $M$ is determined by an object $\EE_M$, the presheaf of spaces on $n$-disks determined by embeddings into $M$, which is an \emph{a priori} weaker invariant of $M$ than $M$ itself. $\EE_M$ encodes the homotopy type of $M$, the homotopy type of all higher configuration spaces ${\conf}_j(M)$ of $M$, and the tangent bundles $T\conf_i(M)$, as well as coherence data relating these, and it would be very interesting to know when this is sufficient to reconstruct $M$.

\subsubsection*{\bf Implementation of $\infty$-categories}
In this work, we use Joyal's {\it quasi-category} model  of $\oo$-category theory \cite{joyal}. 
Boardman and Vogt first introduced these simplicial sets in \cite{bv}, as weak Kan complexes, and their and Joyal's theory has been developed in great depth by Lurie in \cite{topos} and~\cite{dag}, our primary references; see the first chapter of \cite{topos} for an introduction. We use this model, rather than model categories or simplicial categories, because of the great technical advantages for constructions involving categories of functors, which are ubiquitous in this work.

More specifically, we work inside of the quasi-category associated to this model category of Joyal's.  In particular, each map between quasi-categories is understood to be an iso- and inner-fibration; (co)limits among quasi-categories are equivalent to homotopy (co)limits with respect to Joyal's model structure.
As we work in this way, we refer the reader to these sources for $\infty$-categorical versions of numerous familiar results and constructions among ordinary categories.  
To point, we will make repeated use of the $\infty$-categorical adjoint functor theorem (Corollary~5.5.2.9 of~\cite{topos}); the straightening-unstraightening equivalence between Cartesian fibrations over an $\infty$-category $\cC$ and $\Cat$-valued contravariant functors from $\cC$ (Theorem~3.2.0.1 of~\cite{topos}), and likewise between right fibrations over $\cC$ and space-valued presheaves on $\cC$ (Theorem~2.2.1.2 of~\cite{topos}); the $\infty$-categorical version of the Yoneda functor $\cC\to \Psh(\cC)\simeq {\sf RFib}_\cC$ as it evaluates on objects as $c\mapsto \cC_{/c}$ (see~\S5.1 of~\cite{topos}).  

We will also make use of topological categories, such as $\mfld_n$ of $n$-manifolds and embeddings among them. By a functor $\cS \ra \cC$ from a topological category to an $\oo$-category $\cC$ we will always mean a functor $\sN\Sing \cS \ra \cC$ from the simplicial nerve of the ${\sf Kan}$-enriched category obtained by applying the product preserving functor $\Sing$ to the morphism topological spaces. 

The reader uncomfortable with this language can substitute the words ``topological category" for ``$\oo$-category" wherever they occur in this paper to obtain the correct sense of the results, but they should then bear in mind the proviso that technical difficulties may then abound in making the statements literally true. The reader only concerned with algebras in chain complexes, rather than spectra, can likewise substitute ``pre-triangulated differential graded category" for ``stable $\oo$-category" wherever those words appear, with the same proviso.

\subsubsection*{\bf Notation}
\begin{itemize}
\item
$\spaces$ is the $\infty$-category of spaces.  This $\infty$-category has numerous constructions and characterizations: as the ${\sf Kan}$-enriched category of Kan complexes; as the free small colimit completion of the terminal $\infty$-category $\ast$; and as $\infty$-groupoids.  

\item
$\ddisk_n$ and $\dmfld_n$ are ordinary categories of manifolds with embeddings; $\disk_n$ and $\mfld_n$ are topological categories of manifolds with embeddings, where the spaces of embeddings carry the compact-open topology. See Definition \ref{def.ddisk} versus Definition \ref{manifold}.

\item
$\Top(n)$ is the topological group of homeomorphisms of $\RR^n$, endowed with the compact-open topology.

\item
After Definition \ref{bmanifold}, we fix a space $B$ with a map $B\ra \BTop(n)$ and consider $B$-framed $n$-manifolds. Any occurrence of $B$ thereafter refers to this choice.

\item
We use $\colim$ and $\limit$ to denote colimits and limits in $\oo$-categories, which correspond to {\it homotopy} colimits and limits in topological categories or model categories. (In the one or two places where we use a point-set colimit, we employ unmistakeably jarring notation to distinguish the two.)

\item 
For $k$ a ring we use the notation ${\sf Mod}_k$ for the $\infty$-category of $k$-modules.  This is an $\infty$-category associated to the differential graded category of chain complexes over $k$.  (See~\S1.3 of~\cite{dag} for a thorough account.) 
We will sometimes use the notation ${\sf Ch}_k$ for this $\infty$-category, and should $k$ be the integers we drop it from the subscript.

\item
$\cE_n$ will stand for the topological operad of little $n$-cubes, as defined in \cite{bv}.
\item 
$\sC_\ast(X)$ is the singular chains on a topological space $X$.

\item 
${\sf Sym}(V)$ is the free commutative algebra on an object $V$ of a symmetric monoidal $\infty$-category. In a symmetric monoidal $\oo$-category, commutative algebras are equivalent to $\cE_{\oo}$-algebras, so ${\sf Sym}(V)$ is also the free $\cE_{\oo}$-algebra on $V$.

\item For $X\in \cX$ an object of an $\infty$-category, we notate $\cX_{/X}$ and $\cX^{X/}$ for the over- and under-$\infty$-categories.
For $(A\to X)$ and $(B\to X)$ two objects of $\cX_{/X}$, we may denote the space of morphisms from the first to the second as $\Map_{/X}(A,B)$.

\end{itemize}

\begin{ack} 
JF foremost thanks Kevin Costello for many conversations on this subject, which have motivated and clarified this work, from Theorem \ref{homology} to the computations of the following sections, and without which JF likely would not have pursued it. 
Our joint works with Hiro Lee Tanaka build and improve on many of the ideas here, and we thank him for his collaboration. 
JF first learned the basic idea of factorization homology in conversations with Jacob Lurie and Dennis Gaitsgory in 2007, and we have both benefitted greatly from their generosity in sharing many other insights in these intervening years. JF thanks Sasha Beilinson and Mike Hopkins for their great influence which has shaped his thoughts on this subject. We also thank Gr\'egory Ginot and Owen Gwilliam for helpful conversations and Pranav Pandit for comments on an earlier draft of this paper. We thank Amabel Wilson and Theo Johnson--Freyd for correcting the hypotheses in the statement of Proposition~\ref{prop.5.3}.
Finally, we thank the anonymous referee whose careful feedback has considerably improved this article. 
\end{ack}

\section{$B$-framed disks and manifolds}

We now specify the details of our basic objects of study, $n$-manifolds.

\subsection{$B$-framings}
We consider a topological category of $n$-manifolds and embeddings among them.
We use this to consider the tangent classifier, which thereafter offers the notion of a $B$-framing on an $n$-manifold, as well as an $\infty$-category of such.  

\begin{definition}\label{manifold} $\mfld_n$ is the symmetric monoidal topological category for which an object is a topological $n$-manifold that admits a finite \emph{good} cover, which is to say a finite open cover by Euclidean spaces with the property that each non-empty intersection of terms in the cover is itself homeomorphic to a Euclidean space.  
The morphism spaces are spaces of embeddings, endowed with the compact-open topology.
The symmetric monoidal structure is disjoint union.\footnote{Thus, any $M\in \mfld_n$ has finitely many connected components, each of which is the interior of a compact manifold with (possibly empty) boundary. This size restriction is not an essential requirement; since all noncompact manifolds are built as sequential colimits of such smaller manifolds, this smallness condition could be removed and one could instead add to Definition \ref{exc} the requirement that a homology theory preserves sequential colimits.}
\end{definition}

In particular, the mapping space is $\Map_{\mfld_n}(M,N) =  \Emb(M,N)$, the space of embeddings of $M$ into $N$ equipped with the compact-open topology. Note that disjoint union is not the coproduct; $\mfld_n$ has almost no nontrivial colimits.

We will be particularly interested in $n$-manifolds equipped with some extra structure such as an orientation or a framing.  
Structure of this sort can be swiftly accommodated by way of the \emph{tangent classifier}: each $n$-manifold $M$ has a tangent microbundle, and it is classified by a map $\tau_M \colon M \to \BTop(n)$ to the classifying space of the topological group ${\sf Top}(n)$ of self-homeomorphisms of $\RR^n$; see \cite{milnorstasheff}. 
For $B\to \BTop(n)$ a map of spaces, a $B$-framing on $M$ is a homotopy commutative diagram among spaces
\[
\xymatrix{
&&
B  \ar[d]
\\
M \ar[rr]_-{\tau_M} \ar[urr]^-{g}
&&
\BTop(n).
}
\]
\begin{example}
Consider the composite continuous homomorphism
\[
{\sf Top}(n) \xra{~(-)^+~}{\sf Aut}_\ast(S^n) \xra{~\pi_0~} \ZZ/2\ZZ
\]
given by applying 1-point compactification to obtain based homotopy automorphisms of a sphere followed by taking path components.  
For ${\sf STop}(n)\subset {\sf Top}(n)$ the kernel of this homomorphism, a ${\sf BSTop}(n)$-framing on a topological $n$-manifold is precisely an orientation.  

\end{example}
Toward formulating an $\infty$-category of $B$-framed $n$-manifolds, we next explain how to make the tangent classifier continuously functorial among open embeddings.  
We will make ongoing use of the following result of Kister and Mazur.
\begin{theorem}[\cite{kister}]\label{kister}
The continuous homomorphism of topological monoids ${\sf Top}(n) \to {\sf Emb}(\RR^n,\RR^n)$ is a homotopy equivalence.  
\end{theorem}
\noindent
Now, temporarily consider the full $\infty$-subcategory $\cE{\sf uc}_n\subset \mfld_n$ consisting solely of $\RR^n$; this $\infty$-category is that associated to the topological monoid $\Emb(\RR^n,\RR^n)$ of self-embeddings of $\RR^n$.  
We draw an immediate consequence of the Kister--Mazur Theorem.
\begin{cor}\label{euc}
The canonical functor 
\[
\BTop(n) \longrightarrow \cE{\sf uc}_n
\]
is an equivalence of $\infty$-categories.  
In particular, there is a preferred equivalence of $\infty$-categories
\[
\Psh(\cE{\sf uc}_n)~\simeq~\spaces_{/\BTop(n)}
\]
between space-valued presheaves on $\cE{\sf uc}_n$ and spaces over $\BTop(n)$. 

\end{cor}
After Corollary~\ref{euc} we have a tangent classifier, functorial in a coherent homotopy sense, given by the restricted Yoneda functor:
\begin{equation}\label{tau}
\tau\colon \mfld_n \longrightarrow \Psh(\mfld_n) \longrightarrow \Psh(\cE{\sf uc}_n)~\underset{\rm Cor~\ref{euc}}\simeq~\spaces_{/\BTop(n)}~.
\end{equation}
We will postpone to Corollary~\ref{tang-class} justification for this terminology.
To define the $\infty$-category of $B$-framed $n$-manifolds as it is equipped with a symmetric monoidal structure, we record a few standard facts about $\infty$-categories together with an observation about the functor $\tau$.
\begin{lemma}\label{slice-slice}
Let $\cS$ be an $\infty$-category and let $S\in \cS$ be an object.  
\begin{enumerate}

\item For each morphism $S'\to S$ in $\cS$, the canonical functor among over $\infty$-categories $(\cS_{/S})_{/(S'\to S)} \to \cS_{/S'}$ is an equivalence.  

\item Should $\cS$ admit finite coproducts, the over $\infty$-category $\cS_{/S}$ admits finite coproducts and they are preserved by the projection functor $\cS_{/S} \to \cS$.

\end{enumerate}
In addition, the $\infty$-category of symmetric monoidal $\infty$-categories ${\sf Cat}_\infty^\ot$ admits limits and they are preserved by the forgetful functor ${\sf Cat}_\infty^{\ot} \to \Cat$.
\end{lemma}

\begin{proof}
Through the defining adjunctions for over $\infty$-categories, the first assertion follows because, for each $\infty$-category $\cK$, the canonical diagram among $\infty$-categories
\[
\xymatrix{
\{0\}  \ar[r] \ar[d]
&
\{0<1\}   \ar[d]
\\
\cK\star\{0\}   \ar[r]
&
\cK\star \{0<1\}
}
\]
is a pushout; here, for $\cK$ and $\cJ$ $\infty$-categories, 
\[
\cK \star \cI~ :=~ \cK \underset{\cK \times\{0\}\times \cI} \coprod \cK\times\{0<1\}\times \cI \underset{\cK\times \{1\}\times\cI}\coprod \cI
\]
denotes the join of $\infty$-categories.  
The second assertion follows directly from the universal property of coproducts.  
The final assertion follows from Proposition~3.2.2.1 of~\cite{dag}, which in particular gives that, for each Cartesian closed presentable $\infty$-category $\cC$, the forgetful functor from commutative algebras $\Alg_{\sf Com}(\cC^\times) \to \cC$ preserves and creates limits.
Apply this result to the case $\cC = \Cat$.  

\end{proof}

\begin{observation}
Because $\RR^n$ is connected, this tangent classifier $\tau\colon \mfld_n \to \spaces_{/\BTop(n)}$ is symmetric monoidal with respect to coproducts in the codomain.
In other words, $\tau$ carries finite disjoint unions to finite coproducts over $\BTop(n)$.

\end{observation}

\begin{definition}\label{bmanifold}
The symmetric monoidal $\oo$-category $\mfld_n^B$ of $B$-framed topological $n$-manifolds is the limit in the following diagram:
\[\xymatrix{
\mfld^B_n\ar[r]\ar[d]&\spaces_{/B}\ar[d]\\
\mfld_n\ar[r]^-\tau&\spaces_{/\BTop(n)}.
}
\]
\end{definition}

Since passage from $\oo$-categories to their spaces of morphisms preserves limits, there is a corresponding expression for mapping spaces: for two $B$-framed manifolds, $M$ and $N$, the space of $B$-framed embeddings of $M$ to $N$ is the homotopy pullback
\[\xymatrix{
\Emb^B(M,N)\ar[r]\ar[d]&\Map_{/B}(M,N)\ar[d]\\
\Emb(M,N)\ar[r]&\Map_{/\sf BTop(n)}(M,N),\\}\] where $\Map_{/X}(M,N)$ is the space of maps of $M$ to $N$ over $X$, a point of which can be taken to be a map $M\ra N$ and a homotopy between the two resulting maps from $M$ to $X$.

The following assures us that these spaces of $B$-framed embeddings have tractable homotopy types.

\begin{lemma} A $B$-framing $g$ of $\RR^n$ determines a homotopy equivalence $\Emb^B(\RR^n, \RR^n) \simeq \Omega_g B$ of topological monoids, where $\Omega_g B$ is the loop space of $B$ based at the homotopy point $g: \RR^n \ra B$.
\end{lemma}
\begin{proof} By definition, the space $\Emb^B(\RR^n, \RR^n)$ sits in a homotopy pullback square:
\[
\xymatrix{
\Emb^B(\RR^n, \RR^n) \ar[r]\ar[d]& \Map_{/\negthinspace B}(\RR^n, \RR^n)\ar[d]\\
\Emb(\RR^n, \RR^n) \ar[r]&\Map_{/\BTop(n)}(\RR^n, \RR^n).
}
\]
There are evident equivalences of spaces $\Map_{/\BTop(n)}(\RR^n, \RR^n)\simeq \Omega \BTop(n)\simeq \Top(n)$ and likewise $\Map_{/\negthinspace B}(\RR^n, \RR^n)\simeq \Omega_g B$. 
It is standard that the the composite map of spaces
\[
\Top(n)\longrightarrow \Emb(\RR^n,\RR^n)\longrightarrow\Map_{/\BTop(n)}(\RR^n,\RR^n)\simeq \Omega \BTop(n)
\] 
is an equivalence. 
By Kister--Mazur, Theorem \ref{kister}, the first map including $\Top(n)$ into $\Emb(\RR^n, \RR^n)$ is a homotopy equivalence. It follows that the middle map in the above display is also an equivalence of spaces.
We conclude that the bottom horizontal map in the above homotopy pullback square is an equivalence of spaces.
This implies the top horizontal map too is an equivalence of spaces.  

\end{proof}

\subsection{Disks}
We now consider $B$-framed $n$-disks.
In terms of configuration spaces, we identify the maximal $\infty$-subgroupoid of $\disk^B_{n/M}$, $B$-framed $n$-disks embedding into a $B$-framed $n$-manifold.

\begin{definition} The symmetric monoidal $\oo$-category $\disk^B_n$ is the full $\oo$-subcategory of $\mfld^B_n$ whose objects are disjoint unions of $B$-framed $n$-dimensional Euclidean spaces.
\end{definition}

\begin{remark} Consider $\ast\ra \BTop(n)$, the basepoint of $\BTop(n)$. A $\ast$-structure on an $n$-manifold $M$ is then equivalent to a topological framing of the tangent microbundle $\tau_M$ of $M$,\footnote{By smoothing theory, framed topological manifolds are essentially equivalent to framed smooth manifolds except in dimension 4.}
and we denote the associated $\oo$-category of framed $n$-disks as $\disk^{\fr}_n$. This symmetric monoidal $\infty$-category $\disk_n^{\fr}$ is homotopy equivalent to the PROP associated to the $\cE_n$ operad of Boardman-Vogt \cite{bv}.
This follows because the inclusion of rectilinear embeddings as framed embeddings determines a homotopy equivalence $\cE_n(I) \xra\sim \Emb^{\sf fr}(\bigsqcup_I \RR^n, \RR^n)$ from the $I$-ary space of the $\cE_n$ operad.  
\end{remark}

\begin{example}
For $B=\BO(n)$, with the usual map $\BO(n)\ra \BTop(n)$, the $\oo$-category of topological $n$-disks with $\BO(n)$-framings is equivalent to the $\oo$-category of smooth $n$-disks and smooth embeddings, $\disk_n^{\sm}\simeq\disk_n^{{\sf BO}(n)}$. 
These are both equivalent to the PROP associated to the unoriented version of the ribbon, or ``framed," $\cE_n$ operad; see \cite{sw} for a treatment of this operad.\footnote{The historical use of ``framed" here is potentially misleading, since in the ``framed" $\cE_n$ operad the embeddings do not preserve the framing, while in the usual $\cE_n$ operad the embeddings do preserve the framing (up to scale). It might lead to less confusion to replace the term ``framed $\cE_n$ operad" with ``unoriented $\cE_n$ operad."} 
To see these equivalences it is enough to explain why each of the natural maps $\sO(n) \ra \Emb^{\sm}(\RR^n,\RR^n) \ra \Emb^{{\sf BO}(n)}(\RR^n,\RR^n)$ is an equivalence.
Smoothing theory~(\cite{kirbysieb}) gives the equivalence of the second map.  
Via Gram--Schmidt, the inclusion $\sO(n) \xra{\simeq} {\sf GL}(\RR^n)$ is a deformation retraction.  
Conjugation by scaling and translation, $(f,t)\mapsto \bigl(x\mapsto  \frac{f(tx)-f(0)}{t}+f(0)\bigr)$, demonstrates the inclusion ${\sf GL}(\RR^n) \xra{\simeq} \Emb^{\sf sm}(\RR^n,\RR^n)$ as a deformation retraction.
\end{example}

Given a topological space $X$ and a finite cardinality $i$, we let $\conf_i(X)\subset X^i$ denote the subspace of those maps $\{1,\dots,i\}\to X$ which are \emph{injective}.  
This configuration space has an evident free action of the symmetric group $\Sigma_i$.

In the next result, for $M$ a $B$-framed $n$-manifold, we consider the over $\infty$-category 
\[
\disk_{n/M}^B~:=~ \disk^B_n \underset{\mfld_n^B}\times \mfld^B_{n/M}~.
\]
Informally, an object is an embedding $\underset{i}\sqcup \RR^n \hookrightarrow M$ for some $i$.  

\begin{lemma}\label{EE-equivs}
The maximal $\infty$-subgroupoid of $\disk_n^B$ is canonically identified as the space
\[
\underset{i\geq 0} \coprod {B}^i_{\Sigma_i}~\simeq~\bigl(\disk_n^B\bigr)^\sim
\]
where the coproduct is indexed by finite cardinalities and each cofactor is the $\Sigma_i$-homotopy coinvariants of the $i$-fold product of the space $B$.  
In particular, the symmetric monoidal functor $[-]\colon \disk_n^B \to \fin$, given by taking sets of connected components of underlying manifolds, is conservative.  

For $M$ a $B$-framed $n$-manifold,
the maximal $\infty$-subgroupoid of $\disk^B_{n/M}$ is canonically identified as the space
\[
\underset{i\geq 0}\coprod \conf_i(M)_{\Sigma_i}~\simeq~\bigl(\disk^B_{n/M}\bigr)^\sim
\]
where the coproduct is indexed by finite cardinalities, and each cofactor is an unordered configuration space. 

\end{lemma}

\begin{proof}
Lemma~\ref{slice-slice} gives an equivalence $\disk_{n/M}^B\simeq \disk_{n/M}$.
So it suffices to assume the case of an equality $B= \BTop(n)$.

The maximal $\infty$-subgroupoid of $\disk_n$ necessarily lies over the maximal $\infty$-subgroupoid of $\fin$, which is $\underset{i\geq 0}\coprod \sB\Sigma_i$.  
The first assertion will be implied upon verifying, for each $i\geq 0$, that the map
\[
{\Sigma_i\wr {\sf Top}(n)}  \longrightarrow \Emb(\underset{i} \sqcup \RR^n,\underset{i}\sqcup \RR^n)
\]
is weakly homotopy equivalent to an inclusion of connected components.  
This is an immediate consequence of the Kister--Mazur Theorem \ref{kister}.  
Via translation, the inclusion of the subgroup of origin preserving homeomorphisms $\Top_0(n) \xra{\simeq}\Top(n)$ is a homotopy equivalence, and so we recognize further the identification
\[
\underset{i\geq 0}\coprod \sB(\Sigma_i\wr \Top_0(n))\xra{~\simeq~} (\disk_n)^\sim~.
\]

Because the projection $\disk_{n/M}\to \disk_n$ is a right fibration, we recognize the maximal $\infty$-subgroupoid of $\disk_{n/M}$ as
\[
\underset{i\geq 0} \coprod \Emb\bigl(\underset{i}\sqcup \RR^n,M\bigr)_{\Sigma_i\wr \Top_0(n)}\xra{~\simeq~}(\disk_{n/M})^\sim~.
\]
Therefore, the second assertion follows upon showing that the $\Sigma_i$-equivariant continuous map
\[
{\sf ev}_0\colon \Emb\bigl(\underset{i}\sqcup \RR^n,M\bigr)_{\Top_0(n)^i} \longrightarrow \conf_i(M)
\]
is a weak homotopy equivalence.  
This is implied upon showing the homotopy fiber of the continuous map
\[
{\sf ev}_0\colon \Emb(\underset{i}\sqcup \RR^n,M) \longrightarrow \conf_i(M)
\]
is weakly homotopy equivalent to $\Top_0(n)^i$.  
In a standard manner, this map is a Serre fibration, and the fiber over $c\colon \{1,\dots,i\}\hookrightarrow M$ is the space $\Emb_0(\underset{i}\sqcup \RR^n,M)$ of embeddings under $c$.  
Fix such a based embedding $e_0\colon \underset{i}\sqcup \RR^n \hookrightarrow M$.  
So we must show that the composite inclusion
\[
\Top_0(n)^i \hookrightarrow \Emb((0\in \RR^n),(0\in \RR^n))^i\cong \Emb_0(\underset{i}\sqcup \RR^n,\underset{i}\sqcup \RR^n)\xra{-\circ e_0} \Emb_0(\underset{i}\sqcup \RR^n,M)
\]
is a weak homotopy equivalence.

The Kister--Mazur Theorem gives that the first of these maps is a homotopy equivalence. 
The problem is thus reduced to showing that the second of these maps is a weak homotopy equivalence.  
This is the problem of showing that each solid diagram among topological spaces
\[
\xymatrix{
S^{k-1} \ar[rr]^-{f_0}  \ar[d]
&&
\Emb_0(\underset{i}\sqcup \RR^n,\underset{i}\sqcup \RR^n)  \ar[d]
\\
\DD^k  \ar[rr]_-{f}  \ar@{-->}[urr]^-{\widetilde{f}}  
&&
\Emb_0(\underset{i}\sqcup \RR^n,M)
}
\]
admits a filler with respect to which the diagram commutes up to homotopy.
Choose a continuous map $\phi\colon (0,1]\times \RR^n \to \RR^n$ such that $\phi_t$ is an origin preserving open embedding for each $t$, $\phi_1 = {\sf id}_{\RR^n}$, the closure $\overline{\phi_s(\RR^n)}\subset \phi_t(\RR^n)$ whenever $s<t$, and the collection of images $\{\phi_t(\RR^n)\mid 0<t\leq 1\}$ is a basis for the topology about $0\in \RR^n$.  
Choose a continuous map $\DD^k \xra{\epsilon} (0,1]$ for which the restriction $\epsilon_{|S^{k-1}} \equiv 1$ is identically one, and the composition
\[
\DD^k\xra{~f~}\Emb_0(\underset{i}\sqcup \RR^n,M) \xra{~(\underset{i} \sqcup \phi_\epsilon)^\ast~} \Emb_0(\underset{i}\sqcup \RR^n,M) 
\]
factors through $\Emb_0(\underset{i}\sqcup \RR^n,\underset{i}\sqcup \RR^n)$.  
Define $\widetilde{f}$ to be this factorization.  
By construction, the restriction $\widetilde{f}_{|S^{k-1}} = f_0$.  
The map $[0,1]\times\DD^k \to \Emb_0(\underset{i}\sqcup \RR^n,M)$ given by $(t,p)\mapsto f\circ (\underset{i}\sqcup \phi_{t\epsilon +(1-t)})^\ast(p)$ demonstrates a homotopy making the lower triangle commute.

\end{proof}

We conclude this section by justifying the term \emph{tangent classifier} for the functor $\mfld_n \xra{\tau} \spaces_{/\BTop(n)}$ of~(1).  
\begin{cor}\label{tang-class}
The value of the tangent classifier~(\ref{tau}) on a topological $n$-manifold $M$ is the map of spaces $M \xra{\tau_M} \BTop(n)$ classifying the tangent microbundle.  

\end{cor}

\begin{proof}
Recognize $\cE{\sf uc}_n\subset \disk_n$ as the full $\infty$-subcategory consisting of the connected $n$-manifolds.  
Specialize the second statement of Lemma~\ref{EE-equivs} to $i=1$ to obtain an identification
\[
M~\simeq~ \cE{\sf uc}_{n/M}~\simeq~ \Emb(\RR^n,M)_{{\sf Top}(n)} \longrightarrow \BTop(n)
\]
involving the homotopy ${\sf Top}(n)$-coinvariants. 
Manifestly, this map of spaces agrees with the tangent classifier in the case $M=\RR^n$.
By construction, this map is functorial in the argument $M \in \mfld_n$. The general case follows.  
\end{proof}

\subsection{Manifolds with boundary}
We will also employ the category of topological manifolds with boundary.

\begin{definition}  $\mfld_n^\partial$ is the symmetric monoidal topological category of topological $n$-manifolds, possibly with boundary, which have finite good covers by Euclidean spaces $\RR^n$ and upper half spaces $\RR_{\geq 0} \times \RR^{n-1}$. Morphisms are open embeddings which map boundary to boundary. $\disk_n^\partial$ is the full symmetric monoidal topological subcategory of $\mfld_n^\partial$ consisting of finite disjoint unions of $\RR^n$ and $ \RR_{\geq 0}\times \RR^{n-1}$.
\end{definition}

\begin{remark} The category $\disk_n^\partial$ is designed to be minimal with respect to the condition that any finite subset of an $n$-manifold with boundary has an open neighborhood homeomorphic to an object of $\disk_n^\partial$. In particular, the closed $n$-disk $\DD^n$ is consequently not an object of $\disk_n^\partial$.
\end{remark}

The following property is essential.

\begin{prop}\label{product} The functor \[\RR_{\geq 0}\times-:\mfld_{n-1} \longrightarrow \mfld_n^\partial\] is homotopically fully faithful. That is, for every pair of $(n-1)$-manifolds $M$ and $N$, the map \[\Emb(M,N) \longrightarrow \Emb\bigl(\RR_{\geq 0}\times M,\RR_{\geq 0}\times N\bigr)\] is a homotopy equivalence.
\end{prop}
\begin{proof} This follows by the standard method of pushing off to infinity in the $\RR_{\geq 0}$ direction (as in the Alexander trick or the contractibility of foliations on $\RR^n$ up to integrable homotopy). That is, define a deformation retraction onto the subspace $\Emb(M,N)$ by defining for each $t\in [0,1]$ the map
\[h_t: \Emb\bigl(\RR_{\geq 0}\times M, \RR_{\geq 0}\times N\bigr)\longrightarrow \Emb\bigl(\RR_{\geq 0}\times M, \RR_{\geq 0}\times N\bigr)\] by

\[ h_t(g)(s,x)= \left\{ 
  \begin{array}{l l}
    (s, g_0(x)) 
    & 
    \quad \text{for $s< (1-t)^{-1}-1$}
    \\
    g(s +1 - (1-t)^{-1}, x) 
    & 
    \quad \text{for $s\geq (1-t)^{-1}-1$}
  \end{array} \right.
\] 
  where $g_0: M\hookrightarrow N$ is the restriction of $g$ at the value $s=0$.

\end{proof}

 \begin{remark} Together with the Kister--Mazur Theorem~\cite{kister}, the previous proposition implies that the map \[\Top(n-1)\hookrightarrow\Emb(\RR_{\geq 0}\times \RR^{n-1}, \RR_{\geq 0}\times\RR^{n-1})\] is a homotopy equivalence. Likewise, $\disk^\partial_n$ is an unoriented variant of the Swiss cheese operad of Voronov \cite{voronov}. Namely, the framed variant $\disk_n^{\partial, \fr}$ is homotopy equivalent to the PROP associated to the Swiss cheese operad.\end{remark}

\subsection{Localizing with respect to isotopy equivalences}
Here we explain that the $\infty$-category $\disk^B_{n/M}$ is a localization of its un-topologized version $\ddisk^B_{n/M}$ on the collection of those inclusions of finite disjoint unions of disks $U\subset V$ in $M$ that are isotopic to an isomorphism.  
This comparison plays a fundamental role in recognizing certain colimit expressions in this theory, for instance those that support the pushforward formula of~\S\ref{sec.push}.

\begin{definition}\label{def.ddisk}
The ordinary symmetric monoidal category $\dmfld_n$ is that for which an object is a topological $n$-manifold, and a morphism is an open embedding between two such; composition is composition of maps, and the symmetric monoidal structure is given by disjoint union.
Likewise, the ordinary symmetric monoidal category $\ddisk_n\subset \dmfld_n$ is the full subcategory consisting of those topological $n$-manifolds that are homeomorphic to a finite disjoint union of Euclidean spaces.  

\end{definition}

Notice the natural functors
\[
\ddisk_n \longrightarrow \disk_n\qquad \text{ and }\qquad  \dmfld_n \longrightarrow \mfld_n
\]
which are symmetric monoidal.  
We denote the pullback symmetric monoidal $\infty$-categories
\[
\xymatrix{
\ddisk_n^B  \ar[d]  \ar[r]  
&
\disk_n^B  \ar[d]
&&
\dmfld_n^B \ar[r]  \ar[d]
&
\mfld_n^B  \ar[d]
\\
\ddisk_n  \ar[r]
&
\disk_n
&\text{ and }&
\dmfld_n \ar[r]
&
\mfld_n.
}
\]
For each topological $n$-manifold $M$,
we denote the $\infty$-subcategory 
\begin{equation}\label{eqn.I_X}
\cI_M~\subset ~\ddisk^B_{n/M}:= \ddisk_n^B \underset{\dmfld_n^B}\times \dmfld^B_{n/M}
\end{equation}
consisting of the same objects but only those morphisms $(U\hookrightarrow M) \hookrightarrow (V\hookrightarrow M)$ whose image in $\disk^B_{n/M}$ is an equivalence. 

\begin{prop}\label{EEd-vs-EE}
The functor
$
\ddisk^B_{n/M} \longrightarrow \disk^B_{n/M}
$
witness a localization of $\infty$-categories:
\[
\bigl(\ddisk^B_{n/M}\bigr)[\cI_M^{-1}]~\simeq~\disk^B_{n/M}~.
\]

\end{prop}

\begin{proof}
Manifestly, the functor is essentially surjective, and it carries the $\infty$-subcategory $\cI_M$ to the maximal $\infty$-subgroupoid $\bigl(\disk^B_{n/M}\bigr)^\sim$. 
There results a functor $\bigl(\ddisk^B_{n/M}\bigr)[\cI_M^{-1}]\to \disk^B_{n/M}$ from the localization.  
We will argue that this functor is an equivalence by showing it is an equivalence on maximal $\infty$-subgroupoids, then that it is an equivalence on spaces of morphisms.
After Lemma~\ref{slice-slice}, it is enough to consider the case of $B=\BTop(n)$.  
We will adopt the following notation for this proof:
\[
\sD_M~:=~\ddisk_{n/M} \qquad\text{ and }\qquad \cD_M~:=~\disk_{n/M}~.
\]

The maximal $\infty$-subgroupoid of $\sD_M$ is the classifying space $\sB \cI_M$.  
In light of the coproduct expression in Lemma~\ref{EE-equivs}, fix a cardinality $i\geq 0$.  
Consider the full subcategory $\cI_M^{i}\subset \cI_M$ consisting of those $(U\hookrightarrow M)$ for which the cardinality of the connected components $|[U]|=i$.
We thus seek to show that the resulting functor $\cI_M^{i} \to \conf_i(M)_{\Sigma_i}$ witnesses an equivalence from the classifying space.  
We explain the following sequence of weak homotopy equivalences
\begin{eqnarray}
\nonumber
\sB \cI^{i}_M
&
\simeq
&
\underset{(U\hookrightarrow M)\in \cI^{i}_M} \colim~ (\RR^n)^i
\\
\nonumber
&
\xra{\simeq}
&
\underset{(U\hookrightarrow M)\in \cI^{i}_M} {\sf p.s.colim}~ (\RR^n)^i
\\
\nonumber
&
\xra{\cong}
&
\conf_{i}(M)_{\Sigma_{i}}
\end{eqnarray}
where ${\sf p.s.colim}$ denotes the ordinary point-set colimit of topological spaces.  
The first equivalence is formal, because each term in the homotopy colimit is contractible.
By inspection, the category $\cI_M^i$ forms a basis for the standard Grothendieck topology on $\conf_{i}(M)_{\Sigma_{i}}$.  
The third homeomorphism follows.
Because $\conf_i(M)_{\Sigma_i}$ is paracompact, Corollary~1.6 of~\cite{dugger-isaksen} gives that the second map is a weak homotopy equivalence.
In summary, we have verified that the map of maximal $\infty$-subgroupoids
\[
\bigl(\sD_M[\cI_M^{-1}]\bigr)^{\sim} \xra{~\simeq~} \bigl(\cD_M\bigr)^{\sim}
\]
is an equivalence.  

We now show that the functor from the localization induces an equivalence on spaces of morphisms.
Consider the diagram of spaces
\[
\xymatrix{
\bigl(\sD_U[\cI_U^{-1}]\bigr)^{\sim}  \ar[r]  \ar[d]_-{(U\hookrightarrow M)}
&
\cD_U^{\sim}  \ar[d]^-{(U\hookrightarrow M)}
\\
\bigl(\sD_M[\cI_M^{-1}]\bigr)^{(1)} \ar[r]  \ar[d]_-{{\sf ev}_1}
&
\cD_M^{(1)}  \ar[d]^-{{\sf ev}_1}
\\
\bigl(\sD_M[\cI_M^{-1}]\bigr)^{\sim}  \ar[r]
&
\cD_M^{\sim}
}
\]
where a superscript ${}^{(1)}$ indicates a space of morphisms, and the upper vertical arrows are given as $(V\hookrightarrow U)\mapsto \bigl((V\hookrightarrow M)\hookrightarrow (U\hookrightarrow M)\bigr)$.
Our goal is to show that the middle horizontal arrow is an equivalence.
We will accomplish this by showing that the diagram is a map of homotopy fiber sequences, for we have already shown that the top and bottom horizontal maps are equivalences.

The right vertical sequence is a fiber sequence is because such evaluation maps are coCartesian fibrations, in general.  
Then, by inspection, the fiber over $(U\hookrightarrow M)$ is the maximal $\infty$-subgroupoid of the over $\infty$-category $(\cD_M)_{/(U\hookrightarrow M)}$.
This over $\infty$-category is canonically identified as $\cD_U$.  

We now show that the left vertical sequence is a homotopy fiber sequence.  
The space of morphisms $\bigl(\sD_M[\cI_M^{-1}]\bigr)^{(1)}$ is the classifying space of the subcategory of the functor category $\Fun^{\cI_M}\bigl([1],\sD_M\bigr) \subset \Fun\bigl([1],\sD_M\bigr)$ consisting of the same objects but only those natural transformations by $\cI$.  
We claim the fiber over $(U\hookrightarrow M)$ of the evaluation map is canonically identified as in the sequence
\[
\bigl(\sD_U[\cI_U^{-1}]\bigr)^{\sim} \xra{~(U\hookrightarrow M)~}\bigl(\sD_M[\cI_M^{-1}]\bigr)^{(1)} \xra{~{\sf ev}_1~} \bigl(\sD_M[\cI_M^{-1}]\bigr)^{\sim}~.
\]
This claim is justified through Quillen's Theorem B, for the named fiber is the classifying space of the over $\infty$-category $(\cI_M)_{/(U\hookrightarrow M)}$ which is canonically isomorphic to $\cI_U$.  
To apply Quillen's Theorem B we must show that each morphism $(U\hookrightarrow M) \hookrightarrow (V\hookrightarrow M)$ in $\cI$ induces an equivalence of spaces $\sB\bigl((\cI_M)_{/(U\hookrightarrow M)}\bigr) \simeq \sB\bigl((\cI_M)_{/(V\hookrightarrow M)}\bigr)$.  
This map of spaces is canonically identified as the map $\sB \cI_U \to \sB \cI_V$ induced from the inclusion $U\hookrightarrow V$, which, by design, is a bijection on connected components.  
Through the previous analysis of this proof, this map is further identified as the map of spaces $U\hookrightarrow V$.
The Kister--Mazur Theorem~\ref{kister} implies this inclusion $U\hookrightarrow V$ is isotopic to an isomorphism, from which it follows that the map of spaces $\sB\cI_U \xra{\simeq} \sB\cI_V$ is an equivalence. 
We conclude that Quillen's Theorem B applies. 
(For an $\infty$-categorical account of Quillen's Theorem B, see for instance Theorem~5.3 of~\cite{barwick}.)  
\end{proof}

\begin{remark}
Proposition~\ref{EEd-vs-EE} implies that, for each symmetric monoidal $\infty$-category $\cV$, the restriction functor $\Alg_{\disk_n^B}(\cV) \to \Alg_{\ddisk_n^B}(\cV)$ is fully faithful and the essential image consists of the \emph{locally constant}  $\ddisk_n^B$-algebras.  
This result also appears in~\cite{dag} as Theorem~5.4.5.9.

\end{remark}

Proposition~\ref{EEd-vs-EE} offers the following construction.  
\begin{construction}\label{def.f-inverse}
Let $f\colon M \to N$ be a continuous map from a $B$-framed $n$-manifold to a $B'$-framed $k$-manifold, possibly with boundary.
Given a regularity condition on $f$, we will produce a composite map of colored operads
\[
f^{-1} \colon \ddisk^{\partial, B'}_{k/N} \longrightarrow \dmfld^B_{n/M} \longrightarrow \mfld^B_{n/M}~.
\]
The second functor is the standard one.
To describe the first functor we make use of Lemma~\ref{slice-slice} so that we can assume the maps $B\ra\BTop(n)$ and $B'\ra\BTop(k)$ are equivalences.  
For this case, the first functor is given by $(U\hookrightarrow N)\mapsto (U\underset{N}\times M \hookrightarrow M)$, which is evidently functorial as well as monoidal.

Suppose the two restrictions
\[
f_|\colon f^{-1}(N\smallsetminus \partial N) \to N\smallsetminus \partial N\qquad \text{ and }\qquad f_|\colon f^{-1}(\partial N) \to \partial N
\]
are manifold bundles.
Then, by inspection, this functor $f^{-1}$ carries isotopy equivalences to equivalences.
Through Proposition~\ref{EEd-vs-EE}, there results a multi-functor
\begin{equation}\label{f-inv}
f^{-1} \colon \disk^{B'}_{k/N} \longrightarrow \mfld^B_{n/M}~.
\end{equation}

\end{construction}

\section{Homology theories for topological manifolds}

Factorization homology evaluates on a general manifold as the average 
over `factorizations' of the manifold into disks of the values of an $n$-disk algebra on such disks.
We make this precise by defining factorization homology as the left Kan extension of an $n$-disk algebra along the inclusion $\disk_n \hookrightarrow \mfld_n$.  
\\

\noindent
For this section, we fix a symmetric monoidal $\infty$-category $\cV$.

\subsection{Disk algebras}

\begin{definition} 
The $\infty$-category of $\disk_n^B$-algebras in $\cV$
\[
\Alg_{\disk_n^B}(\cV)~:=~\Fun^\ot(\disk^B_n, \cV)
\]
is the $\infty$-category of symmetric monoidal functors.  

\end{definition}

There is the restricted Yoneda functor
\[
\EE\colon \mfld_n^B \longrightarrow  \Psh(\mfld_n^B) \longrightarrow \Psh(\disk_n^B)~.
\]
\begin{definition}\label{coend}
Let $M$ be a $B$-framed $n$-manifold.
Let $A$ be a $\disk_n^B$-algebra in $\cV$.  
\emph{Factorization homology (of $M$ with coefficients in $A$)} is an object of $\cV$ given by either of the equivalent expressions (provided they exist)
\begin{eqnarray}
\nonumber
\int_M A
&
:=
&
\colim\bigl(\disk_{n/M}^B \to \disk_n^B \xra{A} \cV\bigr)
\\
\nonumber
&
\simeq
&
\EE_M\underset{\disk_M^B} \bigotimes A~,
\end{eqnarray}
where the latter is the coend.

\end{definition}

\begin{remark} The fact that one describes factorization homology as either a coend or a left Kan extension is exactly analogous to a more familiar fact about the geometric realizations of a simplicial set $X_\bullet$: one can think of geometric realization as a coend (this is the usual definition, given as a quotient of $\coprod_i X_i \times \Delta^i$), or one can think of it as a left Kan extension (the colimit of the overcategory of simplices in $X_\bullet$ of the functor which sends the simplicial $i$-simplex $\Delta[i]$ to the topological $i$-simplex $\Delta^i$).  
\end{remark}

We will frequently make the following requirement of our target.
\begin{definition}\label{ast} We say symmetric monoidal $\oo$-category $\cV$ is \emph{$\ot$-presentable} if it satisfies both of the following conditions.
\begin{itemize}
\item $\cV$ is presentable: with respect to an understood fixed uncountable cardinal, $\cV$ admits colimits and every object is a filtered colimit of compact objects.  
\item The monoidal structure distributes over small colimits: for each object $V\in \cV$, the functor $V\ot -\colon \cV\to \cV$ carries colimit diagrams to colimit diagrams.
\end{itemize}
\end{definition}

\begin{example}
The Cartesian monoidal $\oo$-category $(\spaces,\times)$ is $\ot$-presentable. Likewise, for $R$ a ring then $\bigl({\sf Mod}_R, \ot\bigr)$, with tensor product relative $R$, is $\ot$-presentable (though the opposite $\bigl({\sf Mod}_R^{\op}, \ot\bigr)$ is not).  

\end{example}

\begin{remark} The results of Section 3 (in particular, the Eilenberg--Steenrod axioms for factorization homology) only require that the monoidal structure distributes over sifted colimits; these results are established in this generality, though with a smooth structure present, in~\S2 of~\cite{aft2}.  
However, the calculations of Section 4 onwards require the monoidal structure to distribute over all colimits, so for simplicity of exposition we enforce this stronger hypothesis throughout.
\end{remark}

The fully faithful symmetric monoidal functor $\iota\colon \disk_n^B \hookrightarrow \mfld_n^B$ gives the restriction functor
\[
\Alg_{\disk_n^B}(\cV)~\longleftarrow~  \Fun^\ot\bigl(\mfld_n^B,\cV\bigr)\colon \iota^\ast~.
\]
The next result identifies factorization homology as a left adjoint to this functor, provided $\cV$ is $\ot$-presentable.  

\begin{prop}\label{as-LKE}
Provided $\cV$ is $\ot$-presentable, there is a left adjoint
\[
\iota_!\colon \Alg_{\disk_n^B}(\cV)~ \rightleftarrows~ \Fun^\ot\bigl(\mfld_n^B,\cV\bigr)\colon \iota^\ast~,
\]
and its value on $A$ evaluates as
\[
\iota_!(A)\colon M\mapsto \int_M A~.
\]
\end{prop}

\begin{proof} 
Presentability of $\cV$ grants the existence of the values $\int_M A$.    
Lemma 4.3.2.13 of~\cite{topos} states that these expressions define a functor, as indicated.
Proposition 4.3.3.7 of~\cite{topos} states that this functor satisfies the universal property of being a left adjoint to the restriction $\iota^\ast$.
We thus have the solid diagram among $\infty$-categories
\[
\xymatrix{
\Alg_{\disk_n^B}(\cV)  \ar[d] \ar@{-->}[rr]
&&
\Fun^\ot\bigl(\mfld_n^B,\cV\bigr)  \ar[d] 
\\
\Fun(\disk_n^B,\cV)  \ar[rr]^-{\iota_!}
&&
\Fun(\mfld_n^B,\cV)
}
\]
in which the downward functors are restriction to underlying $\infty$-categories, these downward functors are fully faithful.
It remains to explain how the $\ot$-presentability of $\cV$ grants the existence of the dashed horizontal functor making the diagram commute.

We must show that, for each symmetric monoidal functor $A\colon \disk_n^B\to \cV$, and for each based map among finite sets $I_+ \xra{f} J_+$, the diagram of $\infty$-categories 
\[
\xymatrix{
(\mfld_n^B)^I  \ar[d]_-{f_\ast}  \ar[rr]^-{(\iota_! A)^I}
&&
\cV^I  \ar[d]^-{f_\ast}
\\
(\mfld_n^B)^J  \ar[rr]^-{(\iota_! A)^J}
&&
\cV^J
}
\]
commutes.
The map $f \colon I_+ \to J_+$ is canonically a composition of a surjective active map $f^{\sf surj}$ followed by an injective active map $f^{\sf inj}$ followed by an \emph{inert} map $f^{\sf inrt}$, and so it is enough to verify commutativity of the above diagram for each such class of maps.
The case of inert maps is obvious, because then $f_\ast$ is projection and $(\iota_! A)^K$ is defined as the $K$-fold product of functors, for $K=I,J$.
The case of injective active maps amounts to verifying that $\iota_! A$ carries each monoidal unit to a monoidal unit.  This follows because $A$ does so and because the over $\infty$-categories $\disk_{n/\emptyset}^B = \{\emptyset\} = \mfld_{n/\emptyset}^B$ consist solely of the empty manifold, which is the monoidal unity.  

The case of surjective active maps follows from the case that $f\colon I_+ \to \ast_+$ is given by $+\neq i\mapsto \ast$, so that $f_\ast = \bigotimes^I$ is the $I$-fold tensor product.
Well, because $A$ is symmetric monoidal, there is a canonical arrow
$
\iota_! A \circ \bigotimes^I 
\longrightarrow 
\bigotimes^I\circ (\iota_! A)^I
$
between functors $(\mfld_n^B)^I \to \cV$ that we will argue is an equivalence.
This arrow evaluates on $(M_i)_{i\in I}$ as the horizontal one in the following natural diagram in $\cV$:
\[
\xymatrix{
\colim\bigl(\disk^B_{n/\underset{i\in I}\bigsqcup M_i} \to \disk_n^B\xra{A} \cV \bigr)  \ar[r]
&
\underset{i\in I}\bigotimes \colim\bigl(\disk^B_{n/M_i} \to \disk_n^B \xra{A} \cV\bigr) 
\\
\colim\Bigl( \underset{i\in I}\prod \disk^B_{n/M_i} \to (\disk_n^B)^I \xra{A^I} \cV^I\xra{\bigotimes^I} \cV \Bigr) \ar[u]^-{(\ast)}  \ar[ur]_-{(\dagger)}
&
.
}
\]
The arrow labeled by~($\dagger$) is an equivalence precisely because $V\ot - \colon \cV \to \cV$ preserves colimits. 
By inspection, the $I$-fold disjoint union functor $\bigsqcup^I\colon \prod_{i\in I} \disk^B_{n/M_i} \xra{\simeq} \disk^B_{n/\underset{i\in I}\bigsqcup M_i}$ is an equivalence between $\infty$-categories, for it is essentially surjective and fully faithful.
It follows that the arrow labeled by~($\ast$) is an equivalence, after observing the following commutative diagram among $\infty$-categories:
\[
\xymatrix{
\underset{i\in I} \prod \disk^B_{n/M_i}  \ar[rr]^-{\bigsqcup^I}  \ar[d]
&&
\disk^B_{n/\underset{i\in I} \bigsqcup M_i}  \ar[d]
\\
(\disk^B_n)^I  \ar[rr]^-{\bigsqcup^I}  \ar[d]_-{A^I}
&&
\disk^B_n  \ar[d]^-A
\\
\cV^I \ar[rr]^-{\bigotimes^I} 
&&
\cV~.
}
\]

\end{proof}

\begin{remark}
Proposition~\ref{as-LKE} implies factorization homology can be expressed as symmetric monoidal left Kan extension, at least when $\cV$ is $\ot$-presentable. This is equivalent to operadic left Kan extension (after parsing Definitions 3.1.1.2 and 3.1.2.2 of \cite{dag}), which is the definition of factorization homology, or topological chiral homology, given by Lurie (Definition~5.5.2.6).

\end{remark}

The following justifies the notational omission of the space $B$ from the notation $\int_MA$.
\begin{prop}\label{without-B}
Given a map $\varphi:B\ra B'$ of spaces over $\BTop(n)$ and $M$ a $B$-framed $n$-manifold and $A$ a $B'$-framed $n$-disk algebra, composition with the map $\varphi$ defines a $B'$-framed $n$-manifold $\varphi M$, and restriction along $\varphi$ defines a $B$-framed $n$-disk algebra $\varphi A$. There is a natural equivalence
\[\int_{\varphi M} A\simeq\int_{M} \varphi A\] between the $B$-framed and $B'$-framed factorization homologies.
\end{prop}
\begin{proof} It suffices to show that the forgetful functor $\disk_{n/M}^B\ra \disk_{n/M}$ is an equivalence. 
By definition, this functor is the projection from the double overcategory:
\[
\disk_{n/M}^B~:=~ (\disk_{n/B})_{/M} \longrightarrow \disk_{n/M}.
\]
This functor is a pullback of the likewise functor $\bigl((\spaces_{/\BTop(n)})_{/B}\bigr)_{/M} \to (\spaces_{/\BTop(n)})_{/M}$, which is an equivalence by Lemma~\ref{slice-slice}.

\end{proof}

\subsection{Factorization homology over oriented $1$-manifolds with boundary}\label{sec.interval}
We show that factorization homology of a closed interval is a two-sided bar construction.

The data of an oriented embedding $U\hookrightarrow [-1,1]$ from a finite disjoint union of oriented intervals, determines a linear ordering of the connected components of $U$.  This is organized as a monoidal functor between $\infty$-operads
\begin{equation}\label{interval-assoc}
\disk^{\partial, {\sf or}}_{1/[-1,1]} \longrightarrow {\sf Assoc}_{\sf RL}
\end{equation}
to the standard multi-category corepresenting the datum of an associative algebra $A$, together with a unital right module $T$ and a unital left module $S$.  
Because the space of oriented embeddings between two oriented intervals is contractible, this functor~(\ref{interval-assoc}) is an equivalence of $\infty$-operads.  
In summary, there is an equivalence of $\infty$-categories
\begin{equation}\label{assoc-disk}
\Alg_{\sf RL}(\cV) \xra{~\simeq~} \Fun^\ot\bigl(\disk_{1/[-1,1]}^{\partial, \sf or}, \cV\bigr)
\end{equation}
where the lefthand $\infty$-category is that of algebras over ${\sf Assoc}_{\sf RL}$.

\begin{example}\label{ex.assoc}
Through~(\ref{assoc-disk}), there is a functor $\Alg^{\sf aug}_{\sf Assoc}(\cV) \to\Fun^\ot\bigl(\disk_{1/[-1,1]}^{\partial, \sf or}, \cV\bigr)$ from augmented associative algebras in $\cV$.  

\end{example}

The next result gives a functor from the simplicial category to 1-disks by a standard construction of counting gaps. Our proof is terse; a lengthier treatment is available in~\S2 of~\cite{aft2}.

\begin{lemma} There is a functor $\bdelta^{\op} \ra \disk^{\partial,{\sf or}}_{1/[-1,1]}$ which is final.
\end{lemma}
\begin{proof}
Let $S\subset \disk^{\partial,{\sf or}}_{1/[-1,1]}$ be the full $\infty$-subcategory of embeddings from oriented intervals that surject onto the endpoints. 
That is, $S$ consists of oriented embeddings among 1-manifolds with boundary of the form $[-1, 0) \sqcup \RR^{\sqcup \ell} \sqcup (0,1]\hookrightarrow [-1,1]$, for $\ell\geq 0$. To show the inclusion $S\subset \disk^{\partial,{\sf or}}_{1/[-1,1]}$ is final, by Quillen's Theorem A, we can show the under $\infty$-category $S^{U/}$ has a contractible classifying space for every finite disjoint union of subintervals $U\subset[-1,1]$. 
This is immediate, because $S^{U/}$ has an initial object, which is a disjoint union of $U$ with connected open neighborhoods of the endpoints not contained in $U$.

Lastly, the result follows because there is an equivalence $S\to \bdelta^{\op}$.
On objects this is given by assigning to
$(U\hookrightarrow [-1,1])$ the set connected components of the complement $\bigl[ [-1,1]\smallsetminus U\bigr]$ together with the linear order inherited from that of $[-1,1]$.
That this assignment defines a functor is routine.
That this functor is an equivalence of $\infty$-categories follows because each comopnent of the space of morphisms of $\disk^{\partial, \sf or}_{1/[-1,1]}$ is contractible.  

\end{proof}

This has an immediate corollary, which states that factorization homology over a closed interval is a two-sided bar construction.

\begin{cor}\label{interval} 
For $\cV$ a symmetric monoidal $\oo$-category which is $\ot$-presentable, and for $R:  \disk^{\partial,{\sf or}}_{1}\ra \cV$ a symmetric monoidal functor, there is a natural equivalence in $\cV$:
\[
R\bigl([-1,1)\bigr)\underset{R\bigl((-1,1)\bigr)}\bigotimes R\bigl((-1,1]\bigr)  \xra{~\simeq~} \int_{[-1,1]} R~.
\]
\end{cor}

\subsection{Homology theories}
We now give our second main definition of this paper, that of a homology theory. First, note that taking products of manifolds defines a functor $\mfld_{n-1}^B\times\mfld_1^{\sf or} \ra \mfld_n^B$, where $\mfld_1^{\sf or}$ is oriented 1-manifolds and 
\[
\mfld_{n-1}^B : = \mfld_{n-1}\underset{\spaces_{/{\sf BTop(n)}}}\times\spaces_{/B}
\] 
is the $\oo$-category of $(n-1)$-manifolds with a $B$-framing on the product of their tangent bundle product with a trivial line bundle. Consequently, any $B$-framed $n$-manifold of the form $M_0\times\RR$, where $M_0$ is an $(n-1)$-manifold, can be given the structure of a $\disk_1^{\sf or}$-algebra in $\mfld_n^B$, since $\RR$ has the structure of a $\disk_1^{\sf or}$-algebra in $\mfld_1^{\sf or}$.

\begin{definition}[Collar-gluing]\label{def.collar-gluing}
A \emph{collar-gluing among $B$-framed $n$-manifolds} is a continuous map
\[
f\colon M \to [-1,1]
\]
to the closed interval for which the restriction $f_|\colon M_{|(-1,1)}\to (-1,1)$ is a manifold bundle.  
We will often denote a collar-gluing $M\xra{f}[-1,1]$ simply as the open cover
\[
M' \underset{M_0\times \RR}\bigcup M''~\cong~M
\]
where $M'= f^{-1}[-1,1)$ and $M''=f^{-1}(-1,1]$ and $M_0 = f^{-1}\{0\}$.  

\end{definition}

\begin{remark}
We find it useful to think of a collar-gluing as the data of a manifold $M$ together with a codimension-1 properly embedded submanifold $M_0\subset M$ that splits the manifold $M$ into two disconnected parts, $M'$ and $M''$. 
Such data is afforded by gluing two manifolds with boundary along a common boundary.  
The actual data of a collar-gluing specifies that named just above, in addition to a bi-collaring $M_0\times \RR\hookrightarrow M$ of $M_0\subset M$.
\end{remark}

Construction~\ref{def.f-inverse} offers, for each collar-gluing $M\xra{f} [-1,1]$ among $B$-framed $n$-manifolds, a monoidal functor
\[
f^{-1}\colon \disk^{\partial,\sf or}_{1/[-1,1]} \longrightarrow \mfld^B_{n/M}~.
\]
In particular, for each symmetric monoidal functor $\mfld^B_n \xra{\cF} \cV$ with $\ot$-presentable codomain, there is a canonical morphism in $\cV$:
\begin{equation}\label{pre-excision}
\cF(M')\underset{\cF(M_0\times\RR)}\bigotimes\cF(M'')~\underset{\rm Cor~\ref{interval}}\simeq~\int_{[-1,1]} \cF\circ f^{-1}~\longrightarrow~ \cF(M)~.
\end{equation}

\begin{definition}\label{exc} 
A symmetric monoidal functor $\cF:\mfld_n^{B}\ra \cV$ satisfies \emph{$\ot$-excision} if, for each collar-gluing $M' \underset{M_0\times \RR}\bigcup M''\cong M$ among $B$-framed $n$-manifolds, the canonical morphism~(\ref{pre-excision}) 
\[
\cF(M')\underset{\cF(M_0\times\RR)}\bigotimes\cF(M'')\xra{~\simeq~}\cF(M)
\] 
is an equivalence in $\cV$.
The $\oo$-category of \emph{homology theories} for $B$-framed $n$-manifolds valued in $\cV$ is the full $\oo$-subcategory 
\[
\bH(\mfld_n^{B},\cV)~\subset~\Fun^\ot(\mfld^B_n, \cV)
\]
consisting of those symmetric monoidal functors that satisfy $\ot$-excision.
\end{definition}

\begin{remark} The behavior of a homology theory with coefficients in $\cV$ depends critically on the symmetric monoidal structure chosen on $\cV$. For instance, for the symmetric monoidal $\infty$-category $(\m_k,\oplus)$ of $k$-modules over a fixed field $k$ with direct sum, a homology theory is forced to be ordinary homology with coefficients in $k$; while for $(\m_k,\ot)$, a homology theories is typically \emph{not} a homotopy invariant of manifolds. \end{remark}

\begin{remark} One can complete the $\oo$-category $\mfld_n$ as follows: first, formally adjoin, for every collar-gluing $M'\underset{M_0\times\RR}\bigcup M''\cong M$, the colimit of the simplicial object ${\sf Bar}_\bullet(M',M_0\times\RR,M'')$; second, Dwyer-Kan localize by forcing the natural map from this new object $|{\sf Bar}_\bullet(M',M_0\times\RR,M'')| \ra M$ to be an equivalence. Denote this completion of the $\oo$-category of manifolds as $\widehat{\mfld}_n$. The completion functor $\mfld_n \ra \widehat{\mfld}_n$ is the universal homology theory: that is, we now have the suggestive equivalence 
\[
\int_M \RR^n~ \simeq ~M
\] 
as objects of $\widehat\mfld_n$ (the lefthand side is not defined in $\mfld_n$). By this universal property, a $\ot$-excisive functor $\mfld_n\ra \cV$ is equivalent to a symmetric monoidal functor $\widehat{\mfld}_n \ra \cV$ that preserves geometric realizations of simplicial objects. \end{remark}

\subsection{Pushforward}\label{sec.push}
We prove that factorization homology satisfies $\ot$-excision.
We do this as an instance of a general paradigm: pushforward.  

The following technical lemma is the crux of the later results of this paper. 
An earlier treatment, not in terms of the pushforward, is in \cite{cotangent}; a generalization of this result to structured stratified spaces is given in \S2 of~\cite{aft2}.
We state the the lemma now, and prove it at the end of this section.

\begin{lemma}\label{excision} For $\cV$ a symmetric monoidal $\oo$-category which is $\ot$-presentable, factorization homology valued in $\cV$ satisfies $\ot$-excision: for any $\disk_n^B$-algebra $A$ in $\cV$, and for any collar-gluing $M'\underset{M_0 \times \RR}\bigcup M''\cong M$ among $B$-framed $n$-manifolds, the canonical morphism in $\cV$
\[
 \int_{M'}A\bigotimes_{\displaystyle\int_{M_0\times\RR}A}\int_{M''}A\xra{~\simeq~} \int_MA
\] 
is an equivalence.
\end{lemma}

We give the following $n=1$ example to indicate the utility of $\ot$-excision as well as some intuition about how factorization homology behaves. 
See~\cite{dag} for a different proof of the following result.

\begin{theorem} 
For an associative algebra $A$ in a symmetric monoidal $\oo$-category $\cV$ which is $\ot$-presentable, there is an equivalence 
\[
\int_{S^1}A~\simeq~ \hh_*(A)
\] 
between the factorization homology of the circle with coefficients in $A$ and the Hochschild complex of $A$.

\end{theorem}
\begin{proof} 
Regard the associative algebra $A$ as a symmetric monoidal functor $A\colon \disk^{\sf or}_1 \to \cV$, as in Section~\ref{sec.interval}.
Consider the standard collar-gluing $\RR\underset{\RR\sqcup \RR}\bigcup \RR \cong S^1$ by hemispheres.
Lemma~\ref{excision}, which states that factorization homology staisfies $\ot$-excision, determines the first of the equivalences in the expression:
\[\int_{S^1}A
~\simeq~ \int_{\RR}A\underset{{\displaystyle\int_{{S^0\times\RR}}\! A}}\bigotimes\int_{\RR}A~\simeq~ A\underset{A\ot A^{\op}}\bigotimes A~\simeq~ \hh_*(A)~. \]
The second equivalence is by inspecting values, and the final equivalence is definitional. 

\end{proof}

The next definition makes use of the multi-functor $f^{-1} \colon \disk^{\partial, \sf or}_{k/N} \to \mfld^B_{n/M}$ of Construction~\ref{def.f-inverse} associated to each continuous map $M \ra N$ for which each of the restrictions, $M_{|N\smallsetminus \partial N} \to N\smallsetminus \partial N$ and $M_{|\partial N} \to \partial N$, are manifold bundles.  
\begin{definition}\label{disk-f} Let $M$ be an $B$-framed $n$-manifold, and let $N$ be an oriented $k$-manifold, possibly with boundary. For $f:M\ra N$ a map such that the restrictions of $f$ over each of the interior of $N$ and of the boundary of $N$ is a fiber bundle, the $\oo$-category $\disk_f$ is the limit of the diagram among $\infty$-categories
\[
\xymatrix{
\disk^B_{n/M}\ar[dr]&&\ar[ld]^{{\sf ev}_0}{\sf Ar}(\mfld^B_{n/M})\ar[rd]_{{\sf ev}_1}&&\ar[dl]^{f^{-1}}\disk_{k/N}^{\partial,{\sf or}}\\
&\mfld^B_{n/M}&&\mfld^B_{n/M}
}
\] 
where ${\sf Ar}(\mfld^B_{n/M})$ is the $\oo$-category of functors $[1]\to \mfld^B_{n/M}$. 
\end{definition}

Informally, $\disk_f$ consists of compatible triples $(V,U, V\hookrightarrow f^{-1}U)$ such that: $U$ is an open submanifold of $N$ that is homeomorphic to a disjoint union of Euclidean spaces; $V$ is an open submanifold of $M$ that is homeomorphic to a disjoint union of Euclidean spaces; the embedding $V\hookrightarrow f^{-1}U$ is compatible with the embeddings $f^{-1}U\hookrightarrow M$ and $V\hookrightarrow M$. The relevance of the $\oo$-category $\disk_f$ is the following technical result.

\begin{lemma}\label{final} In the situation of Definition~\ref{disk-f}, the functor ${\sf ev}_0:\disk_f \ra \disk_{n/M}^B$ is final.
\end{lemma}

\begin{proof} 
After Lemma~\ref{slice-slice}, it will suffice to prove the result for the case $B = \BTop(n)$, and so we omit $B$ from the notation and discussion.  
The functor ${\sf ev}_0$ is a Cartesian fibration of $\oo$-categories. Thus, to check finality, by Lemma 4.1.3.2 of \cite{topos}, it suffices to show that, for each $(V\hookrightarrow M)\in \disk_{n/M}$, the fiber $\infty$-category ${\sf ev}_0^{-1}V$ has contractible classifying space. 
That is, we show that the $\oo$-category $(\disk_{k/N})^{V/}$, of $k$-disks $U$ in $N$ equipped with an embedding $V\hookrightarrow f^{-1}U$, has a contractible classifying space. 

There is an identification of spaces
\[
\sB\bigl(\disk_{k/N}\bigr)^{U/}~{}~\simeq~{}~\underset{(V\hookrightarrow N)\in \disk_{k/N}}\colim ~\Map_{\mfld_{n/M}}(U,f^{-1}V)~.  
\]
Formally, the sequence of maps
\[
\Map_{\mfld_{n/M}}(U,f^{-1}V)~\longrightarrow~\Map_{\mfld_n}(U,f^{-1}V)~\longrightarrow~\Map_{\mfld_n}(U,M)
\]
is a fiber sequence (here the fiber is taken over any implicit morphism $U\hookrightarrow M$, thereby giving meaning to the lefthand space).  
So we seek to show the map from the colimit
\[
\underset{(V\hookrightarrow N)\in \disk_{k/N}}\colim \Map_{\mfld_n}(U,f^{-1}V)~{}~\longrightarrow~{}~ \Map_{\mfld_n}(U,M)
\]
is an equivalence of spaces. 
We recognize this map of spaces as the map of fibers over $U\in \disk_n$ of the map of right fibrations over $\disk_n$:
\[
\underset{(V\hookrightarrow N)\in \disk_{k/N}}\colim \disk_{n/f^{-1}V} \longrightarrow \disk_{n/M}~.
\]
Being right fibrations, it is enough to show that this functor is an equivalence on maximal $\infty$-subgroupoids.
Using Lemma~\ref{EE-equivs} which identifies these maximal $\infty$-subgroupoids, this is the problem of showing, for each finite set $J$, that the map of spaces
\[
\underset{(V\hookrightarrow N)\in \disk_{k/N}}\colim \conf_J\bigl(f^{-1}V\bigr)_{\Sigma_J} ~\longrightarrow~ \conf_J(M)_{\Sigma_J}
\]
is an equivalence.

Lemma~\ref{EEd-vs-EE} implies the functor $\ddisk_{k/N} \to \disk_{k/N}$ is final, and so the forgetful map
\[
\underset{(V\hookrightarrow N)\in \ddisk_{k/N}} \colim \conf_J\bigl(f^{-1}V\bigr)_{\Sigma_J} 
~\xra{~\simeq~}~
\underset{(V\hookrightarrow N)\in \disk_{k/N}}\colim \conf_J\bigl(f^{-1}V\bigr)_{\Sigma_J} 
\]
is an equivalence of spaces.  
Now notice that, for each $(V\hookrightarrow N)\in \ddisk_{k/N}$, the map $\conf_J(f^{-1}V) \to \conf_J(M)$ an open embedding.  Also, for each point $c\colon J\hookrightarrow M$ the image $f\bigl(c(J)\bigr)\subset N$ has cardinality at most $J$.  So there is an object $(V\hookrightarrow N)$ of $\ddisk_{k/N}$ whose image contains the subset $f\bigl(c(J)\bigr)$.  We see then that the collection of open embeddings
\[
\Bigl\{ \conf_J(f^{-1}V)_{\Sigma_J} \hookrightarrow \conf_J(M)_{\Sigma_I}\mid (V\hookrightarrow N)\in \ddisk_{k/N} \Bigr\}
\]
forms an open cover.

Because $N$ is a manifold, the collection of open embeddings from Euclidean spaces into $N$ form a basis for the topology of $N$.  
It follows that the collection of (at most) $|J|$-tuples of disjoint open disks in $N$ forms an open cover of $N$ in such a way that any finite intersection of such is again covered by such.  
This is to say that this collection of open embeddings forms a hypercover of $\conf_J(M)_{\Sigma_J}$.
Corollary 1.6 of~\cite{dugger-isaksen} gives that the map
\[
\underset{(V\hookrightarrow N)\in \ddisk_{k/N}}\colim \conf_J\bigl(f^{-1}V\bigr)_{\Sigma_J} 
~\xra{~\simeq~}~
\conf_J(M)_{\Sigma_J}
\]
is an equivalence of spaces, which completes the proof.

\end{proof}

Here is an important technical property of the $\infty$-category of disks over a manifold.  (See also Proposition~5.5.2.16 of~\cite{dag}.)
\begin{cor}\label{disk-sifted}
For $M$ a $B$-framed $n$-manifold, the $\oo$-category $\disk_{n/M}^B$ is sifted.
\end{cor}
\begin{proof}
The $\oo$-category $\disk_{n/M}^B$ is evidently nonempty, as it contains the object $(\emptyset \hookrightarrow M)$. 
We must then prove that the diagonal functor $\disk_{n/M}^B \to \disk_{n/M}^B\times \disk_{n/M}^B$ is final.
This diagonal functor fits into a diagram among $\infty$-categories
\[
\xymatrix{
\disk_{\nabla} \ar[rr]^-{{\sf ev}_0}  \ar[d]_-{{\sf ev}_1}
&&
\disk_{n/M\sqcup M}^B  
\\
\disk_{n/M}^B  \ar[rr]^-{\sf diag}
&&
\disk_{n/M}^B\times \disk_{n/M}^B    \ar[u]^-{\simeq}_{\sqcup}
}
\]
that we now explain.
The upper left $\infty$-category is that of Definition~\ref{disk-f} applied to the fold map $\nabla \colon M \sqcup M \to M$; it is equipped with the indicated projection functors.
The right vertical arrow is induced by the symmetric monoidal structure on $\mfld_n^B$, which is disjoint union.  
This right vertical arrow is an equivalence; an inverse is given by declaring its projection to each factor to be given by intersecting with the corresponding cofactor of the disjoint union.
Therefore, to prove that the diagonal functor is final it is sufficient to prove that both of the projection functors ${\sf ev}_1$ and ${\sf ev}_0$ are final.  
The finality of ${\sf ev}_0$ is Lemma~\ref{final}.

We explain that ${\sf ev}_1$ is final.
Note that the functor $\nabla^{-1} \colon \disk_{n/M}^B \to \mfld_{n/M\sqcup M}^B$ factors through the full $\infty$-subcategory $\disk_{n/M\sqcup M}^B$.  
As so, there is a canonical identification between $\infty$-categories
\[
\disk_\nabla~\simeq~ \disk_{n/M}^B\underset{\disk_{n/M\sqcup M}^B}\times  {\sf Ar}(\disk_{n/M\sqcup M}^B)
\]
over $\disk_{n/M}^B$.
Through this identification, the composite functor 
\[
\disk_{n/M}^B \xra{~\nabla~} \disk_{n/M\sqcup M}^B \xra{~\sf const~} {\sf Ar}(\disk_{n/M\sqcup M}^B)
\]
determines a right adjoint to the functor ${\sf ev}_1$.  
The finality of ${\sf ev}_1$ thereby follows.

\end{proof}

The pushforward property for factorization homology immediately follows from Lemma~\ref{final}, as the next result articulates. We will use the notation
\[
f_\ast A:\xymatrix{\disk_{k/N}^{\partial,{\sf or}}\ar[r]^{f^{-1}}& \mfld_{n/M}^B\ar[r]^-{\int A}& \cV}
\] 
for the composite functor, where $f^{-1}$ is as in Construction~\ref{def.f-inverse}.

\begin{prop}\label{pushforward} Let $M$ be a $B$-framed $n$-manifold, $N$ an oriented $k$-manifold, possibly with boundary, and $f:M\ra N$ a map which fibers over the interior and boundary of $N$. For $A$ a $B$-framed $n$-disk algebra in $\cV$, a symmetric monoidal $\oo$-category which is $\ot$-presentable, then the canonical morphism in $\cV$ \[\int_{N}f_\ast A\xra{~\simeq~}
\int_MA\] is an equivalence.
\end{prop}

\begin{proof} 
After Proposition~\ref{without-B}, we can assume that $B$ is equivalent to $\BTop(n)$, and so we omit it from the notation and discussion.
We will explain the string of canonical equivalences in $\cV$:
\begin{eqnarray}
\nonumber
\int_N f_\ast A
&
\underset{\rm Def~\ref{coend}}\simeq
&
\colim_{U\in\disk_{k/N}^{\partial,{\sf or}}}f_\ast A(U)
\\
\nonumber
&
\underset{{\rm Def~}f_\ast}{\simeq}
&
\colim_{U\in\disk_{k/N}^{\partial,{\sf or}}}\int_{f^{-1}U}A
\\
\nonumber
&
\underset{\rm Def~\ref{coend}}{\simeq}
&
\colim_{U\in\disk_{k/N}^{\partial,{\sf or}}}\colim_{V\in \disk_{n/f^{-1}U}}A(V)
\\
\nonumber
&
\underset{(1)}{\simeq}
&
\colim_{(U,V)\in \disk_f} A(V)
\\
\nonumber
&
\underset{(2)} {\xra{\simeq}}
&
\colim_{V\in \disk_{n/M}} A(V)
\\
\nonumber
&
\underset{\rm Def~\ref{coend}}\simeq
&
\int_M A~.
\end{eqnarray}
The only equivalences that are not definitional are~(1) and~(2).  
The equivalence~(2) is a direct application of Lemma~\ref{final}, which states that the functor ${\sf ev}_0\colon \disk_f \to \disk_{n/M}$ is final.  
Consider the left Kan extension (non-commutative) diagram among $\infty$-categories:
\[
\xymatrix{
\disk_f  \ar[r]  \ar[d]^-{{\sf ev}_1}  \ar[r]^-{{\sf ev}_0}
&
\disk_{n/M}  \ar[r]
&
\disk_n  \ar[r]^-{A}
&
\cV
\\
\disk^{\partial,\sf or}_{k/N}   \ar[urrr]_-{\sf LKan}
&&&
,
}
\]
which exists because $\cV$ is presentable.  
By construction, the functor ${\sf ev}_1\colon \disk_f \to \disk^{\partial, \sf or}_{k/N}$ is a coCartesian fibration.  
In particular, for each object $U\in \disk^{\partial, \sf or}_{k/N}$, the inclusion of the fiber into the over $\infty$-category
\[
{\sf ev}_1^{-1}(U) \longrightarrow (\disk_f)_{/U}
\]
is final.  
Therefore, the value of ${\sf LKan}$ on $U\in \disk^{\partial, \sf or}_{k/N}$ is the colimit over the fiber: 
\[
\colim_{V\in \disk_{n/f^{-1}U}} A(V)
~\underset{\rm Def~\ref{disk-f}}\simeq~
\colim_{V\in {\sf ev}_1^{-1}U} A(V)
\xra{~\simeq~}
{\sf LKan}(U)~.
\]
So the colimit of ${\sf LKan}$ is the codomain of~(1).
The equivalence~(1) follows from Proposition~4.3.3.7 of~\cite{topos}, which implies the colimit of a left Kan extension agrees with the colimit because they both satisfy the same universal property.  
 
\end{proof}

\begin{proof}[Proof of Lemma~\ref{excision}] 
Given a collar-gluing $f:M\ra [-1,1]$, there are canonical morphisms in $\cV$ \[\int_{M'}A\bigotimes_{\displaystyle\int_{M_0\times\RR}A}\int_{M''}A\longrightarrow \int_{[-1,1]} f_\ast A\longrightarrow \int_M A~~. \]
The lefthand morphism is an equivalence by Lemma \ref{interval}, which shows that factorization homology over the closed interval is equivalent to the bar construction, and inspection of the functor $f_\ast A$.  
The righthand morphism is an equivalence by Proposition \ref{pushforward}, which grants that factorization homology pushes forward along $M\xra{f}[-1,1]$.  

\end{proof}

\subsection{Homology theories}
We give the following characterization, \`a la Eilenberg--Steenrod, for factorization homology; this is the central conceptual result of this paper.

\begin{theorem}\label{homology} For $\cV$ a symmetric monoidal $\oo$-category which is $\ot$-presentable, there is an equivalence
\[\xymatrix{
\displaystyle\int: \Alg_{\disk_n^B}(\cV)\ar@<-.5ex>[r] &\bH(\mfld_n^B, \cV):{\sf ev}_{\RR^n}\ar@<-.5ex>[l]\\}\]
between $\disk^B_n$-algebras in $\cV$ and homology theories of $B$-framed $n$-manifolds with coefficients in $\cV$. This equivalence is implemented by the factorization homology functor $\int$ and the functor of evaluation on $\RR^n$.

\end{theorem}

\begin{proof}
Proposition~\ref{as-LKE} recognizes factorization homology as a symmetric monoidal left Kan extensions, thereby implementing the left adjoint in an adjunction 
\[
i_!\colon \Alg_{\disk_n^B}(\cV) \rightleftarrows  \Fun^\ot\bigl( \mfld_n^B, \cV\bigr)\colon i^\ast~.
\]
The unit of this adjunction is an equivalence because $\disk_n^B\ra \mfld_n^B$ is fully faithful, and the Kan extension along a fully faithful functor restricts as the original functor.  
The counit of this adjunction evaluates on a symmetric monoidal functor $\cF$ as a morphism $\int\! A \ra \cF$, where $A=\cF_{|\RR^n}$ is the $\disk^B_n$-algebra defined by the values of $\cF$ on disjoint unions of $B$-framed Euclidean $n$-spaces. 
It remains to verify that this counit is an equivalence.

Since both $\cF$ and $\int\! A$ are symmetric monoidal and agree on $\RR^n$, the map $\int_MA\ra \cF(M)$ is an equivalence for $M$ isomorphic to a disjoint union of Euclidean spaces, $M\cong \bigsqcup_I \RR^n$. 
Using induction, we will now see that the values of $\cF$ and $\int\! A$ agree on thickened spheres $S^k\times \RR^{n-k}$, the base case of $k=0$ just having been shown. In the inductive step, assume the result for $S^{i-1}\times\RR^{n-i+1}$. 
Choose a standard collar-gluing $S^i \xra{f} [-1,1]$ with $S^{i-1} =f^{-1}(0)\subset S^i$ an equator.  
There results a collar-gluing of $S^i\times\RR^{n-i}$.
For $\cF$ a homology theory, we obtain the equivalence $\int_{S^i\times\RR^{n-i}}A\simeq \cF(S^i\times\RR^{n-i})$ via the intermediate equvialences 
{\Small
\[\underset{S^i\times\RR^{n-i}}\int\! A \simeq \underset{\RR_{-1}^i\times\RR^{n-i}}\int\! A\underset{\underset{S^{i-1}\times\RR^{n-i+1}}\int \negthinspace A}\bigotimes \ \underset{\RR_{+1}^i\times\RR^{n-i}} \int\! A \ \simeq \ \cF(\RR_{-1}^{i}\times\RR^{n-i})\underset{\cF(S^{i-1}\times\RR^{n-i+1})}\bigotimes \cF(\RR_{+1}^{i}\times\RR^{n-i})\simeq \cF(S^i\times\RR^j)
\]
}

\noindent where the first equivalence is by the $\ot$-excision property of factorization homology, the last equivalence is by the assumption that $\cF$ is a homology theory, and the middle equivalence is by induction.  

We now restrict to the case of $B$-framed $n$-manifolds where $n$ is not equal to 4. By the handlebody theory for topological manifolds (\cite{kirbysieb} for $n>5$, \cite{quinn} for $n=5$, and \cite{moise} for $n=3$) all such manifolds admits a handle decomposition. We now prove the result outside dimension 4 by induction on the handle decomposition. The base case is assured. To verify the inductive step, let $M$ be obtained from $M_0$ by adding a handle of index $q+1$. Therefore $M$ can be expressed as a collar-gluing $M\cong M_0  \underset{S^q\times\RR^{n-q}}\bigcup \RR^n$, where $\RR^n$ is an open neighborhood of the $(q+1)$-handle in $M$. The values $\cF$ and $\int\! A$ agree on the three constituent submanifolds of $M$, and they both satisfy $\ot$-excision, so the values $\cF(M)\simeq \int_M A$ are equivalent.

This leaves the case of topological 4-manifolds, which do not admit handle decompositions in general. 
Because both $\cF$ and $\int_{\!} A$ are symmetric monoidal, we can reduce to the case that $M$ is connected.  Now, any connected topological 4-manifold $M$ admits a smooth structure on the complement $M\smallsetminus \{x\}$ of a point $x\in M$, \cite{quinn}. Consequently, $M\smallsetminus \{x\}$ admits a handle decomposition, which can be constructed from any Morse function on $M\smallsetminus\{x\}$, and the preceding argument thereby implies the equivalence $\cF(M\smallsetminus \{x\})\simeq \int_{M\smallsetminus \{x\}}A$. Applying the $\ot$-excision property to the collar-gluing $M\smallsetminus \{x\}\underset{S^{n-1}\times\RR}\bigcup \RR^n\cong M$, since $\cF$ and $\int\! A$ agree on the constituent submanifolds, we obtain the equivalence $\cF(M)\simeq \int_M A$. Therefore every homology theory $\cF$ for $n$-manifolds is equivalent to factorization homology with coefficients in $\cF(\RR^n)$.

\end{proof}

We record here this technical comparison of factorization homology with different target $\oo$-categories. (This result is also Proposition~5.5.2.17 of~\cite{dag}.)
\begin{lemma}\label{switch} For $G: \cV\ra\cV'$ a symmetric monoidal functor between $\ot$-presentable $\infty$-categories whose restriction to underlying $\infty$-categories preserves geometric realizations, there is a canonical equivalence $\int \circ G \xra{\simeq} G\circ \int$ of functors $\Alg_{\disk_n^B}(\cV) \ra \Fun^\ot(\mfld_n^B,\cV')$.
\end{lemma}
\begin{proof} The two functors agree on disjoint unions of Euclidean spaces, so it suffices to check that $G\int A$ is a homology theory with values in $\cV'$.
This is immediate by the assumption on $G$.  
\end{proof}

A result identical to Theorem \ref{homology} holds for topological $n$-manifolds with boundary.

\begin{theorem} For $\cV$ a symmetric monoidal $\oo$-category which is $\ot$-presentable, and for $B\to \BTop(n)$ a map of spaces, there is an equivalence between $\infty$-categories
\[\xymatrix{
\displaystyle\int: \Alg_{\disk_n^{\partial, B}}(\cV)\ar@<-.5ex>[r] &\bH(\mfld_n^{\partial, B}, \cV):{\sf ev}_{\RR^n, \RR^{n-1}\times[0,1)}\ar@<-.5ex>[l]\\}\]
from $\disk^{\partial, B}_n$-algebras in $\cV$ and homology theories of $B$-framed $n$-manifolds with coefficients in $\cV$. This equivalence is implemented by the factorization homology functor $\int$ and evaluation on $B$-framed Euclidean $n$-spaces and half-spaces. 
\end{theorem}
\begin{proof} 
After Proposition~\ref{without-B} it is enough to consider the case where $B$ is equivalent to $\BTop(n)$, and so we omit it from the notation and discussion.
Let $\cF:\mfld_n^{\partial}\ra \cV$ be a symmetric monoidal functor satisfying the $\ot$-excision condition, and let $A$ be the restriction of $\cF$ to $\disk_n^\partial$. For a manifold with boundary $\overline M$, we prove that the canonical morphism $\int_{\overline M}A \ra \cF(\overline M)$ is an equivalence. By $\ot$-excision applied to the collar-gluing $\partial \overline M \times [0,1)\underset{\partial \overline M\times\RR}\bigcup M\cong \overline{M}$, where $M$ is the interior of $\overline M$, we obtain a diagram in $\cV$
\[
\xymatrix{
\displaystyle\underset{{\partial \overline M \times[0,1)}}{\int}A\underset{\underset{\partial \overline M\times\RR}\int A}\bigotimes\int_{M}A\ar[r]\ar[d]&\cF(\partial \overline M \times[0,1))\underset{\cF (\partial \overline M\times\RR)}\bigotimes\cF(M)\ar[d]\\
\displaystyle\int_{\overline M}A\ar[r]&\cF(\overline M)
}
\] in which the vertical morphisms are equivalences. It now suffices to show that the top horizontal morphism is an equivalence. The equivalences of $\int_M A\ra \cF(M)$ and $\int_{\partial \overline M\times \RR}A\ra \cF(\partial\overline M\times \RR)$ are given by Theorem \ref{homology}; the last equivalence $\int_{\partial \overline M \times[0,1)} A \ra \cF({\partial \overline M \times[0,1)})$ by follows by Theorem \ref{homology} and Proposition \ref{product}.
\end{proof}

\begin{remark} There is an analogous theorem available for stratified spaces proved in \S2 of~\cite{aft2}, where embeddings are conically smooth and preserve the stratifications. The previous theorems hold if the $\ot$-presentable condition is weakened to the condition that the monoidal structure distributes over sifted colimits. 
\end{remark}

The following example describes how factorization homology specializes to usual homology.

\begin{example}\label{ES} Let $\cV^\oplus$ be either the $\oo$-category of chain complexes or of spectra equipped with the direct sum monoidal structure. Since every object $V$ therein has an essentially unique morphism $V\oplus V\ra V$, there is an equivalence $\Alg_{\disk_n^{\sf fr}}(\cV^\oplus)\simeq \cV$. The factorization homology of a framed $n$-manifold $M$ with coefficients in $V$ is then equivalent to $\int_MV \simeq \sC_\ast(M,V)$, or $\Sigma^\infty_\ast M\ot V$ for spectra, the stabilization of $M$ smashed with $V$. There is a natural functor $\varinjlim \mfld_n^{\fr}\ra \spaces^{\sf fin}$, and this functor is an equivalence because: it is fully faithful since $\varinjlim_k \Emb^{\fr}(M\times\RR^k,N\times\RR^k) \ra \Map(M\times\RR^\infty,N\times\RR^\infty) \simeq\Map(M,N)$ is a weak homotopy equivalence for every $M$ and $N$; it is essentially surjective since every finite CW complex $X$ can be embedded into $\RR^m$ for $m$ sufficiently large, and thus it is homotopy equivalent to a framed $n$-manifold, namely an open regular neighborhood of the embedding. 
Theorem \ref{homology} thereby specializes to the formulation of the Eilenberg--Steenrod axioms given in the introduction. If one sets $\cV$ to be the opposite ${\sf Ch}^{\op}$, then one likewise recovers the Eilenberg--Steenrod axioms for cohomology.

\end{example}

To this point, we have worked with topological manifolds and embeddings with the compact-open topology, but other choices could have been made, for instance, to work with smooth manifolds, or to have regarded the embeddings spaces as discrete.
We next remark on as to how these alternate choices play out.

\begin{remark} One could replace the $\oo$-category of topological $n$-manifolds and	 embeddings with that of smooth $n$-manifolds and smooth embeddings, $\mfld^{\sm}_n$, or piecewise linear $n$-manifolds and piecewise linear embeddings, $\mfld^{\sf PL}_n$, and Theorem~\ref{homology} is still valid. The proof, in fact, is even simpler, as the existence of handlebody structures is much easier than in the case of topological manifolds. However, a consequence of smoothing theory \cite{kirbysieb} is an equivalence \[\mfld^{\sm}_n\simeq \mfld_n^{{\sf BO}(n)}\] between smooth $n$-manifolds and $\BO(n)$-framed topological $n$-manifolds, so long as $n$ is not equal $4$, so nothing new is obtained by considering smooth or piecewise linear manifolds rather than $B$-framed topological manifolds. In the case of smooth 4-manifolds, there is still an equivalence $\disk_4^{\sm} \simeq \disk_4^{{\sf BO}(4)}$, and so combining the smooth and topological versions of Theorem \ref{homology} gives an equivalence \[\bH(\mfld_n^{\sm},\cV)\simeq \bH(\mfld_4^{{\sf BO}(4)},\cV)\] between homology theories for smooth 4-manifolds and homology theories for ${\sf BO}(4)$-framed topological 4-manifolds. Since the ${\sf BO}(4)$-framing on the tangent microbundle is only a very weak measure of a smooth structure in dimension 4 (e.g., there is a single ${\sf BO}(4)$-framing of $\RR^4$, in contrast to the uncountably many smooth structures), this form of factorization homology is not a refined invariant of smooth 4-manifolds. 
\end{remark}

\medskip 

In the subsequent sections, we will be solely concerned with the homology theories of Definition \ref{exc}. There do, however, exist very interesting functors in $\Fun^\ot(\mfld_n, \cV)$ which do not satisfy the $\ot$-excision property. In \cite{qcloops} the authors were particularly concerned with one such construction: given a stack $X$ over $k$, one can define a functor $\mfld_n \ra {\sf Stacks}\ra \m_k$ given by sending a manifold $M$ to the cotensor  with $X$, $M\rightsquigarrow X^M$, and then taking sheaf cohomology of the structure sheaf of this stack. 
In the case of the circle, $M=S^1$, this gives the Hochschild homology of $X$: $\cO(X^{S^1})\simeq \hh_*(X)$. As soon as $X$ is nonaffine, this construction will generically fail to satisfy $\ot$-excision. While the cotensor only depends on the homotopy type of $M$, as we shall see in Proposition \ref{tensor}, it has a more refined generalization taking as input a derived stack defined over $n$-disk algebras, rather than commutative algebras, as in \cite{thez}.

\begin{definition}\label{stack} 
Let $\cV$ be a symmtric monoidal $\infty$-category which is $\ot$-presentable.
For a $B$-framed $n$-manifold $M$ and a functor $X:\Alg_{\disk^B_n}(\cV)\ra \spaces$, the factorization homology of $M$ with coefficients in $X$ is the object in $\cV$
\[
\int_MX := \limit_{A\in{\sf Aff}^{\op}_{/X}}\int_M A
\] 
where ${\sf Aff} \simeq \Alg_{\disk^B_n}(\cV)^{\op}$ is the image of the Yoneda embedding in $\Fun\bigl(\Alg_{\disk_n^B}(\cV),\spaces\bigr)$.
\end{definition}

Intuitively, the object $\int_M X$ is $\Gamma(X,\int_M\cO)$, the global sections of the presheaf on $X$ obtained by applying factorization homology of $M$ to the structure sheaf of $X$. From the vantage offered by Costello and Gwilliam in \cite{kevinowen}, this generalization of factorization homology serves as a candidate for the structure of observables in a topological quantum field theory which is not necessarily perturbative, a direction we will pursue in future work.

\section{Nonabelian Poincar\'e duality}

Applying Theorem \ref{homology}, we offer a slightly different perspective, and proof, of the nonabelian Poincar\'e duality of Salvatore \cite{salvatore}, Segal~\cite{segallocal}, and Lurie \cite{dag}, which calculate factorization homology with coefficients in iterated loop spaces as a compactly supported mapping space.

\begin{definition} For a space $B$, the $\infty$-category $\Space_B$ is that of retractive spaces over $B$. The $\oo$-category $\Space^{\geq n}_B$ is the full $\oo$-subcategory of $\Space_B$ consisting of those $X\rightleftarrows B$ for which the retraction is $n$-connective, that is, $\pi_* X \ra \pi_* B$ is an isomorphism for $*< n$ with any choice of base-point of $B$.\end{definition}

Equivalently, an object $X\in \Space_B^{\geq n}$ may be thought of as a fibration $X\ra B$ with a distinguished section for which, for each $b\in B$, the fiber $X_b$ is $n$-connective.

\begin{definition} For $X\in \Space_B$ and $M \ra B$ a space over $B$, the space $\Gammac(M,X)$ of compactly supported sections of $X$ over $M$ is the subspace of $\Map_{/B}(M,X)$ consisting of those maps $f\colon M\to X$ over $B$ for which there is a compact subspace $K\subset M$ with the property that the restriction  factors through the section: $f_{|M\smallsetminus K} : M \to B\to X$.
\end{definition}

For $B\ra \BTop(n)$ as before, note that $\Gammac(-,X)$ defines a covariant functor $\mfld_n^B \ra \Space$.
By inspection, this functor carries finite disjoint unions to finite products of spaces, which is to say that $\Gammac(-,X)$ is symmetric monoidal with respect to the Cartesian monoidal structure on the $\infty$-category of spaces.

To easily state the next result, Theorem~\ref{spaces}, we introduce some terminology.
Each point $g\in B$ determines a symmetric monoidal functor $\disk_n^{\fr} \to \disk_n^B$.
Thereafter, each symmetric monoidal functor $A\colon \disk_n^B \to \spaces$ determines the associative monoid
\[
\pi_0(A,g)\colon \disk_1^{\fr} \xra{\RR^{n-1}\times -} \disk_n^{\fr} \longrightarrow \disk_n^B \xra{~A~} \spaces \xra{~\pi_0~} {\sf Set} ~.
\]
We say $A$ is \emph{group-like} if, for each $g\in B$, this monoid $\pi_0(A,g)$ is a group.  

\begin{example}
In the case $B\simeq \ast$ so that a $B$-framing is a framing in the standard sense, a symmetric monoidal functor $A\colon \disk_n^{\sf fr} \to \spaces$ is the data of an $\cE_n$-algebra, and this $\cE_n$-algebra is group-like in the standard sense if and only if $A$ is group-like in the the sense just above.  

\end{example}

\begin{theorem}\label{spaces} 
The functor $\Gammac: \Space_B \ra \Fun^\ot(\mfld_n^B, \Space)$ restricts as a fully faithful functor 
\[
\Space_B^{\geq n} ~\hookrightarrow~ \bH(\mfld_n^B,\Space)
\] 
from $n$-connective retractive spaces over $B$ to homology theories. The essential image consists of those $\cF$ for which the restriction $\cF_{|\disk_n^B}$ is group-like.
\end{theorem}

The following is the core technical detail in the proof of Theorem~\ref{spaces}, that the assignment of compactly supported sections $\Gammac(-,X)$ is $\ot$-excisive in our sense, provided the retraction $X\ra B$ is sufficiently connected. The fullness of the functor above is a parametrized form of May's theorem from \cite{may}, identifying $n$-connective objects as $\disk^{\fr}_n$-algebras, which is Theorem 5.1.3.6 of \cite{dag}.
 \
\begin{lemma}\label{gamma-c} 
Let $M$ be $B$-framed manifold, equipped with a collar-gluing $M'\underset{M_0\times\RR}\bigcup M''\cong M$. 
Let $X\in \Space^{\geq n}_B$ be an $n$-connective space over $B$. There is a natural weak homotopy equivalence 
\[
\Gammac(M',X)\underset{\Gammac(M_0\times\RR,X)}\times\Gammac(M'',X)~\simeq~\Gammac(M, X)
\] 
between the quotient of the product $\Gammac(M',X)\times\Gammac(M'',X)$ by the diagonal action of $\Gammac(M_0\times\RR,X)$ and the space of compactly supported sections of $X$ over $M$.

\end{lemma}

\begin{proof}

Since $M_0\hookrightarrow M$ is a proper embedding, a compactly supported section over $M$ can be restricted to obtain a compactly supported section over $M_0$, as well as over $M\smallsetminus M'$ and over $M\smallsetminus M''$.
Namely, there is a diagram among spaces of compactly supported sections
\[
\xymatrix{
\Gammac(M',X) \times \Gammac(M'',X)  \ar[dr]  \ar[rr]  \ar[dd]
&
&
\Gammac(M'',X)  \ar[d]
\\
&
\Gammac(M,X)  \ar[r]  \ar[d]  \ar[dr]
&
\Gammac(M\smallsetminus M',X)  \ar[d]
\\
\Gammac(M',X)  \ar[r]
&
\Gammac(M\smallsetminus M'',X)  \ar[r]
&
\Gammac(M_0,X).
}
\]
By inspection, the bottom horizontal sequence is a fiber sequence, as is the right vertical sequence, as is the diagonal sequence.
Also, the inner square is pullback because $M'\underset{M_0\times \RR}\bigcup M''\cong M$ is a pushout.  
Because $M_0\subset M$ is equipped with a regular neighborhood, these fiber sequences are in fact Serre fibration sequences, and so the inner square is a weak homotopy pullback square.  
In particular, there is a right homotopy coherent action of $\Omega \Gammac(M_0,X)$ on $\Gammac(M',X)$, a left homotopy coherent action of $\Omega \Gammac(M_0,X)$ on $\Gammac(M'',X)$, and a continuous map of topological spaces
\begin{equation}\label{gammas}
\Gammac(M',X)\underset{\Omega \Gammac(M_0,X)} \times \Gammac(M'',X)\longrightarrow \Gammac(M,X)
\end{equation}
from the balanced homotopy coinvariants.  
Because $X\to B$ is $n$-connective and $M_0$ is $(n-1)$-dimensional, the base $\Gammac(M_0,X)$ is connected.
It follows that the map~(\ref{gammas}) is in fact a weak homotopy equivalence.  
The assertion follows after the canonical identification $\Omega \Gammac(M_0,X) \cong \Gammac(M_0\times \RR,X)$ as group-like $\cE_1$-spaces.

\end{proof}

As a consequence, we recover the following theorem of Salvatore \cite{salvatore}, Segal~\cite{segallocal}, and Lurie \cite{dag}.

\begin{cor}[Nonabelian Poincar\'e duality]\label{non-abel} 
For any $X$ in $\Space_B^{\geq n}$, with associated $\disk^B_n$-algebra $\Omega_B^nX\in \Alg_{\disk^B_n}(\Space)$, there is a natural equivalence 
\[
\int_M \Omega_B^n X ~ \simeq ~\Gammac(M,X)
\] 
between the factorization homology of a $B$-framed $n$-manifold $M$ with coefficients in $\Omega^n_{B}X$ and the space of compactly supported sections of $X$ over $M$.
\end{cor}
\begin{proof} We apply Theorem \ref{homology}: Since $\Gammac(-,X)$ is a homology theory, it is equivalent to factorization homology with coefficients in $\Gammac(\RR^n,X)$, which is identified as the $n$-fold loop space of the fiber of the map $X\ra B$, $\Gammac(\RR^n,X)\simeq \Omega^n_BX$.
\end{proof}

This result specializes to Poincar\'e duality between twisted homology and compactly supported cohomology, as we will now explain.
Given $X\simeq \BTop(n)\times K(A,i)$ a product with an Eilenberg-MacLane space, then $\Omega^n_{\BTop(n)} X \simeq \Mapc(\RR^n, K(A,i)) \simeq K(A,i-n)$ is an $n$-disk algebra in spaces, where the multiplication is the usual group structure on $K(A,i-n)$ but is equipped with a nontrivial action of $\Top(n)$. We then have an equivalence of spaces
\[\int_M \Mapc(\RR^n, K(A,i)) \simeq \Mapc(M,K(A,i))\]  which is the space level version of the equivalence $\sH_*^\tau(M,A[i-n])\simeq \sH^*_{\sf c}(M, A[i])$, obtained by applying $\Omega^\infty$ to the spectrum level equivalence given by Atiyah duality. Given an $A$-orientation of $M$, then one can additionally untwist the lefthand side, as usual.

\begin{remark} The factorization homology $\int_M\Omega^n_BX$ is built from configuration spaces of disks in $M$ with labels defined by $X\to B$, and the preceding result thereby has roots in the configuration space models of mapping spaces dating to the work of Segal, May, McDuff and others in the 1970s, see \cite{segal}, \cite{may}, \cite{mcduff}, and \cite{bodig}. Factorization homology is not a generalization 
of the classical configuration spaces with labels, as described in~\cite{bodig}, because the configuration space with labels in $X$ models a mapping space with target the $n$-fold suspension of $X$, rather than into $X$ itself. Instead, factorization homology generalizes the configuration spaces with {\it summable} or {\it amalgamated} labels of Salvatore \cite{salvatore} and Segal \cite{segallocal}.

\end{remark}

\begin{proof}[Proof of Theorem~\ref{spaces}]
Corollary~\ref{non-abel} identifies the functor $\Gamma_{\sf c} \colon \spaces_B^{\geq n} \to  \bH(\mfld_n^B,\Space)$ as the composition $\Gamma_{\sf c} \colon \spaces_B^{\geq n} \xra{\Omega^n_B} \disk_n^B \xra{\int}  \bH(\mfld_n^B,\Space)$.
Theorem~\ref{homology} gives that $\int$ is fully faithful, so it remains to argue that $\Omega_B^n$ is fully faithful with essential image the group-like $\disk_n^B$-algebras in spaces.  
This is immediate because, for instance, $\spaces_{/B}$ is an $\infty$-topos (Theorem 5.1.3.6 of \cite{dag}).  

\end{proof}

\section{Commutative algebras, free algebras, and Lie algebras}

Previously, we have described factorization homology for $n$-disk algebras in spaces, chain complexes, and spectra, when the monoidal structure is given by products, and the resulting homology theories give rise to twisted mapping spaces and usual homology theories. Factorization homology behaves very differently, and with greater sensitivity to manifold topology, when the monoidal structure on chain complexes or spectra is given by tensor product or smash product -- this case is closest to the physical motivation given in the introduction. 
We will consider this case in this section, focusing on some of the most common classes of $n$-disk algebra structures, which are either commutative, freely generated by a module, or freely generated by a Lie algebra.

\subsection{Factorization homology with coefficients in commutative algebras}
We begin by examining commutative algebras in $\cV$, otherwise known as $\cE_\infty$-algebras in $\cV$.
Note first that a commutative algebra in $\cV$ is equivalent to a symmetric monoidal functor ${\sf Fin}\ra \cV$ from finite sets with disjoint union.
So restriction along the connected components functor $[-]: \disk^B_n\ra\disk_n \xra{[-]} {\sf Fin}$ defines a forgetful functor 
\[
{\sf fgt}\colon \Alg_{\com}(\cV)\longrightarrow \Alg_{\disk^B_n}(\cV)~.
\] 
We have the following consequence of $\ot$-excision, where $\cV$ is a symmetric monoidal $\infty$-category which is $\ot$-presentable.
To phrase this result we utilize that the $\infty$-category $\Alg_{\com}(\cV)$ is tensored over spaces: 
\[
\spaces\times \Alg_{\com}(\cV) \xra{\ot} \Alg_{\com}(\cV)~,\qquad (X,A)\mapsto \colim\bigl(X\to \ast \xra{\{A\}} \Alg_{\com}(\cV)\bigr)~.
\]

\begin{prop}\label{tensor} 
The following diagram among $\infty$-categories commutes:
\[
\xymatrix{
\mfld^B_n \times \Alg_{\com}(\cV) \ar[r]^-{U\times {\sf id}}\ar[d]^-{{\sf id}\times {\rm fgt}} 
& 
\Space \times \Alg_{\com}(\cV)\ar[r]^-\ot &\Alg_{\com}(\cV)\ar[d]
\\
\mfld^B_n \times\Alg_{\disk^B_n}(\cV)\ar[rr]_\int &&\cV
}
\]
where $U$ is the underlying space functor and the right downward arrow is the standard forgetful functor.
In particular, there is a natural equivalence \[\int_M A~\simeq~ M\ot A\] \noindent between the factorization homology of $M$ with coefficients in $A$ and the tensor of the commutative algebra $A$ with the underlying space of $M$.
\end{prop}
\begin{proof} 
The functor $-\ot A\colon \mfld_n^B \to \cV$ carries each contractible manifold to the underlying object of the commutative algebra $A$.  
For $(A_i)_{i\in I}$ a finite sequence of commutative algebras in $\cV$, the $I$-fold coproduct in $\Alg_{\com}(\cV)$ is the pointwise tensor product $\underset{i\in I} \bigotimes A_i$ (see Proposition~3.2.4.7 of~\cite{dag}).  
It follows that this functor $-\ot A$ is symmetric monoidal.  
From the defining expression of factorization homology as a colimit, there results a natural transformation
\[
\int_- A \longrightarrow -\ot A
\]
between symmetric monoidal functors $\mfld_n^B\to \cV$, which evaluates as an equivalence on objects of $\disk_n^B$.
Lemma~\ref{excision} grants that the domain of this natural transformation satisfies $\ot$-excision.
Because a collar-gluing $M'\underset{M_0\times \RR}\bigcup M'' \cong M$ determines a pushout of underlying spaces $M'\underset{M_0}\coprod M'' \simeq M$, the codomain of this natural transformation too satisfies $\ot$-excision.  
That the natural transformation evaluates on each $B$-framed $n$-manifold $M$ as an equivalence then follows by induction on a handle decomposition on $M$.  

\end{proof}

In other words, the factorization homology $\int_M A$ has a natural structure of a commutative algebra when $A$ is commutative, and this commutative algebra has a universal property: for each commutative algebra $C$ in $\cV$ there is a natural equivalence from the space of commutative algebra maps
\[
\Map_{\com}\bigl(\int_MA,C\bigr)~\simeq ~\Map_{\com}(A,C)^M~,
\]
to the space of maps from $M$ to the space of commutative algebra maps.
By formal properties of left adjoints and tensors, this has the immediate corollary.

\begin{cor} 
For each symmetric monoidal $\infty$-category $\cV$ which is $\ot$-presentable, there is a natural equivalence in $\cV$: \[\int_M \sym(V) ~\simeq ~\sym(M\ot V)~. \]
\end{cor}

In particular, if $\cV$ is the $\infty$-category of chain complexes with tensor product, then there is an equivalence $\int_M\sym(V) \simeq \sym (\sC_\ast(M,V))$ for each chain complex $V$. We now push the above result slightly further for the two special classes of commutative algebras arising from the cohomology of spaces and the cohomology of Lie algebras.
The study of the latter has benefitted greatly from conversations with Kevin Costello and Dennis Gaitsgory, and a full development of these ideas will amount to a forthcoming work.  

\begin{prop}\label{prop.5.3} 
Let $M$ be an $n$-manifold, and let $X$ be a nilpotent $n$-connective space of finite type over $R$ such that $\pi_nX$ is finite. There is a natural equivalence of chain complexes \[\int_M \sC^\ast(X, R)~ \simeq~ \sC^\ast(X^M,R)\] between the factorization homology of $M$ with coefficient in the $R$-cohomology of $X$ and the $R$-cohomology of the space of maps from $M$ to $X$.
\end{prop}
\begin{proof}
The two sides are evidently equivalent in the case where $M$ is homeomorphic to $\RR^n$, so to establish the result it suffices, as usual, to check by induction over a handle decomposition of $M$. Given a handle decomposition $N\underset{S^{k}\times\RR^{n-k}}\bigcup\RR^n\cong M$, we have a homotopy pullback diagram of spaces
\begin{equation}\label{nilpotent-square}
\xymatrix{
X^M\ar[r]\ar[d]&\ar[d] X^{\RR^n}\\
X^{N}\ar[r]&X^{S^{k}\times\RR^{n-k}}\\}
\end{equation}
which gives rise to a natural map in $R$-homology
\[\sC_\ast(X^{M},R)\longrightarrow \sC_\ast(X^{N},R)\underset{\sC_\ast(X^{S^k\times\RR^{n-k}},R)}\bigotimes\sC_\ast(X^{\RR^n},R)\] from the homology of the mapping spaces to the cotensor product of the comodules $\sC_\ast(X^{N},R)$ and $\sC_\ast(X^{\RR^n},R)$ over the coalgebra $\sC_\ast(X^{S^k\times\RR^{n-k}},R)$. This map is an equivalence exactly if the homological Eilenberg--Moore, or Rothenberg--Steenrod, spectral sequence for this homotopy Cartesian diagram converges. By Dwyer \cite{dwyer}, the convergence of this Eilenberg--Moore spectral sequence is assured if the base $X^{S^k\times\RR^{n-k}}$  is connected and the action
\[
\pi_1 \bigl(X^{S^k\times\RR^{n-k}}, f\bigr)
\circlearrowright
\pi_\ast\bigl({\sf fiber}_f(X^{\RR^n}\ra X^{S^k\times\RR^{n-k}})\bigr)
\]
is nilpotent for a choice of basepoint $f \in X^{S^k\times\RR^{n-k}}$.
Since $X$ is $n$-connective, for $k<n$ any map $f:  S^k \ra X$  is nullhomotopic, and therefore the base $X^{S^k\times\RR^{n-k}} \simeq X^{S^k}$ is connected.
We can thus take $f$ to be the constant map valued at the basepoint of $X$, and so identify ${\sf fiber}_f(X^{\RR^n}\ra X^{S^k\times\RR^{n-k}}) \simeq \Omega^{k+1}X$.
We now show the action of $\pi_1 \bigl(X^{S^k}\bigr)$ on $\pi_\ast \Omega^{k+1} X$ is nilpotent. 
Consider the fiber sequence $\Omega^{k+1} X \to X^{S^k} \xra{{\sf ev}_\ast} X$.
This fibration admits a section, given by the constant maps. 
Consequently, there is an identification as a semi-direct product:
\[
\pi_1\bigl(X^{S^k}\bigr)
\cong
\pi_1X \ltimes\pi_1\Omega^k X
\cong
\pi_1X \ltimes\pi_0 \Omega^{k+1} X
~.
\]
Through this identification, the action of $\pi_1 \bigl(X^{S^k}\bigr)$ on $\pi_\ast \Omega^{k+1} X$ is the unique action that extends the standard actions of $\pi_1X$ and of $\pi_0 \Omega^{k+1}X$ on $\pi_\ast \Omega^{k+1}X$.
By assumption, the action of $\pi_1X$ on $\pi_{\ast+k+1}X \cong \pi_\ast \Omega^{k+1} X $ is nilpotent.
In the case that $k=0$, the same assumption grants that the action of $\pi_0 \Omega^{k+1} X$ on $\pi_\ast \Omega^{k+1}X$ is nilpotent. 
In the case that $k>0$, the action of $\pi_0 \Omega^{k+1} X$ on $\pi_\ast \Omega^{k+1}X$ is automatically nilpotent due to commutativity. 
Nilpotence of the action of $\pi_1 \bigl(X^{S^k}\bigr)$ on $\pi_\ast \Omega^{k+1} X$ follows. Consequently, the natural map in $R$-homology above is an equivalence.

The remainder of this argument is checking that we have imposed sufficient finiteness conditions to ensure the convergence in cohomology as well as homology. Dualizing, we obtain an equivalence \[\Bigl(\sC_\ast(X^{N},R)\underset{\sC_\ast(X^{S^k\times\RR^{n-k}},R)}\bigotimes\sC_\ast(X^{\RR^n},R)\Bigr)^\vee \overset{\sim}\longrightarrow \sC^\ast(X^{M},R)~. \] Since $\pi_nX$ is finite, the mapping space $X^K$ has finitely many components for any $n$-dimensional finite CW complex $K$. Because the source spaces, $M$, $N$, $S^k\times\RR^{n-k}$, and $\RR^n$, all have have the homotopy types of finite $n$-dimensional CW complexes, we obtain that all these mapping spaces have finitely many components. Since they are additionally finite CW complexes and $X$ is finite type, the homology groups of the mapping spaces $\sH_i(X^K,R)$ are finite rank over $R$, and therefore $\sC_*(X^K,R)$ is its own double dual: the map $\sC_*(X^K,R)\ra \sC^*(X^K,R)^\vee$ is an equivalence. Likewise, there is an equivalence between the dual of the tensor product and the cotensor product \[\Bigl( \sC^\ast(X^{N},R)\underset{\sC^\ast(X^{S^k\times\RR^{n-k}},R)}\bigotimes\sC^\ast(X^{\RR^n},R)\Bigr)^\vee \simeq \sC_\ast(X^{N},R)\underset{\sC_\ast(X^{S^k\times\RR^{n-k}},R)}\bigotimes\sC_\ast(X^{\RR^n},R)\simeq \sC_\ast(X^M,R)\]
-- this can be seen by commuting duality with the colimit to obtain a limit of a cosimplicial object, then comparing termwise. 
Continuing, one then concludes the equivalence 
\[
\sC^\ast(X^M,R) \simeq \sC^\ast(X^N,R)\underset{\sC^\ast(X^{S^k\times\RR^{n-k}},R)}\bigotimes \sC^\ast(X^{\RR^n},R)~,
\]
thereby finishing the proof.

 \end{proof}

\begin{remark} See \cite{gtz1} for a closely related approach to the study of mapping spaces, in which one approaches the cohomology of a mapping space as a Hochschild homology-type invariant of the cohomology of the target.
\end{remark}

\subsection{Factorization homology with coefficients in free $n$-disk algebras}
We next turn to the factorization homology of free $n$-disk algebras, a topic studied in more detail in~\S2 of~\cite{aft2}. 

Denote by $\free_n(V)$ the augmented $n$-disk algebra freely generated by $V\in \cV$, regarded as a trivial $\Top(n)$-module. Let $\conf_i(M,\partial M)$ denote the quotient of $\conf_i(M)$ by the subspace of all configurations in which at least one point lies in the boundary of $M$.

\begin{prop}\label{free} Let $M$ be an $n$-manifold, possibly with boundary. Let $V\in \cV$ be an object of a symmetric monoidal $\oo$-category which is $\ot$-presentable. There is an equivalence 
\[
\int_M \free_n(V) ~\simeq~ \coprod_{i\geq 0} \conf_i(M,\partial M)\underset{\Sigma_i}\ot V^{\ot i}
\] 
between the factorization homology of an $n$-manifold $M$, possibly with boundary, with coefficients in $\free_n(V)$ and the coproduct of the configuration spaces of $M$ labeled by $V$ quotient the subspace where at least one point lies in the boundary of $M$. 
\end{prop}

The argument below is a special case of one in \cite{af1}.

\begin{proof}  
The following equivalences  \[\int_M \free_n(V) \simeq \colim_{U\in\disk^{\partial}_{n/M}}\coprod_{i\geq 0}\conf_i(U,\partial U)\underset{\Sigma_i}\ot V^{\ot i}\simeq \coprod_{i\geq 0}\colim_{U\in\disk^{\partial}_{n/M}}\conf_i(U,\partial U)\underset{\Sigma_i}\ot V^{\ot i}\] follow from the commutativity of colimits. 
To conclude the result it therefore suffices to show that for each $i$ the canonical morphism
\[
\colim_{U\in\disk^\partial_{n/M}}\conf_i(U,\partial U)\underset{\Sigma_i}\ot V^{\ot i}
\longrightarrow 
\conf_i(M,\partial M)\underset{\Sigma_i}\ot V^{\ot i}
\] 
in $\cV$ is an equivalence. 
By the assumed distributivity in the $\ot$-presentability condition, this follows if the natural $\Sigma_i$-equivariant map of $\Sigma_i$-spaces 
\begin{equation}\label{compose-coend}
\colim_{U\in\disk^\partial_{n/M}}\conf_i(U, \partial U)\longrightarrow \conf_i(M, \partial M)
\end{equation}
is an equivalence, which we now show.

We first consider the case that the boundary of $M$ is empty, so that the natural $\Sigma_i$-equivariant map $\conf_i(M) \xra{\cong} \conf_i(M,\partial M)$ is a homeomorphism, and the natural functor $\disk_{n/M}\xra{\simeq}\disk^\partial_{n/M}$ is an equivalence of $\infty$-categories.  
In this case we are to show that the $\Sigma_i$-equivariant map of $\Sigma_i$-spaces
\[
\colim_{U\in \disk_{n/M}} \conf_i(U)\longrightarrow \conf_i(M)
\]
is an equivalence. 
After Proposition~\ref{EEd-vs-EE}, it is enough to show that the $\Sigma_i$-equivariant map of $\Sigma_i$-topological spaces
\begin{equation}\label{hyper}
\underset{U\in {\sf Disk}_{n/M}}\colim \conf_i(U) \longrightarrow \conf_i(M)
\end{equation}
is a weak homotopy equivalence from the homotopy colimit.  
Each map $\conf_i(U)\to \conf_i(M)$ comprising this homotopy colimit is an open embedding.
Also, for each 
element $\{1,\dots,i\}\xra{c} M$ of $\conf_i(M)$, choosing mutually disjoint Euclidean neighborhoods about each $c(j)\in M$ demonstrates that $c$ lies in the image of at least one such open embedding. 
Therefore this augmented diagram is an open cover of $\conf_i(M)$.
This open cover of $M$ has the property that each finite intersection of its terms is covered by terms contained in this finite intersection.
This is to say that this open cover of $\conf_i(M)$ is in fact a hypercover.
That the map~(\ref{hyper}) is a weak homotopy equivalence follows from Corollary 1.6 of~\cite{dugger-isaksen}.

Now suppose $\partial M$ is not empty.  
Fix a collar-neighborhood $\partial M \times \RR_{\geq 0} \hookrightarrow M$.
Such a collar-neighborhood determines the top horizontal arrow in the diagram of topological spaces
\[
\xymatrix{
\underset{\emptyset \neq I \subset \{1,\dots,i\}}\colim \conf_{\{1,\dots,i\}\smallsetminus I}(\partial M)\times \conf_{I}(\mathring{M})  \ar[rr]  \ar[d]
&&
\conf_i(\mathring{M})  \ar[d]
\\
\ast \ar[rr]
&&
\conf_i(M,\partial M)
}
\]
which commutes up to homotopy -- here, the homotopy colimit is indexed by the opposite of the poset of non-empty subsets of $\{1,\dots,i\}$.
This collar-neighborhood also gives that this diagram is a weak homotopy pushout.
The result for this case of non-empty boundary thus follows from the previous case of empty boundary applied to $\partial M$ and to $\mathring{M}$, using that homotopy colimits commute with one another.  

\end{proof}

\begin{remark} The preceding result has as a consequence that factorization homology is not a homotopy invariant of a closed $n$-manifold, since the homotopy type of configuration spaces is known to be sensitive to simple homotopy equivalence by \cite{simple}.
\end{remark}

From Proposition \ref{free} and some reasoning on stable splittings of configuration spaces, one can deduce the following result. For the previous proposition, we required the monoidal structure of $\cV$ to distribute over colimits; for convenience, we next assume the underlying $\infty$-category of $\cV$ is stable.

\begin{prop}\label{splits} 
Let $\cV$ be a symmetric monoidal $\infty$-category which is $\ot$-presentable and whose underlying $\infty$-category is stable.
For any $V\in\cV$ and any nonnegative integer $m<n$, there is an equivalence in $\cV$:
\[
\int_{S^{m}\times\RR^{n-m}}\free_{n}(V)~ \simeq~ \free_n(V) \ot \free_{n-m}(\Sigma^m V)~. 
\]
\end{prop}

\begin{proof} 
We first consider the case that $\cV = \bigl({\sf Spectra}, \wedge\bigr)$ and $V=\Sigma^\infty X$ is the suspension spectrum of a pointed connected space $X$. 
There is a natural equivalence of spaces 
\[
\int_{S^m\times\RR^{n-m}}\Omega^n\Sigma^n X~\underset{\rm Cor~\ref{non-abel}}\simeq~ (\Omega^{n-m}\Sigma^nX)^{S^m}~\simeq~ \Omega^n\Sigma^n X\times \Omega^{n-m}\Sigma^n X
\] 
where the first equivalence is by nonabelian Poincar\'e duality and the second is by the standard trivialization of each fiber sequence $\Map_*(K,G)\ra G^K\ra G$ whose base is equipped with the structure of a group-like $\cE_1$-space.
(It is this second step that requires the strict inequality $m<n$.) 
Passing to suspension spectra, Proposition~\ref{free} begets the further equivalence
\[\bigvee_{k\geq 0} \conf_k(S^m\times\RR^{n-m})\underset{\Sigma_k}\ot\Sigma^{\oo}X^{\ot k}\simeq \Bigl(\bigvee_{i\geq 0} \conf_i(\RR^n)\underset{\Sigma_i}\ot\Sigma^{\oo}X^{\ot i} \Bigr)\ot\Bigl(\bigvee_{j\geq 0} \conf_j(\RR^{n-m})\underset{\Sigma_j}\ot\Sigma^{\oo}(\Sigma^m X)^{\ot j}\Bigr)~.\]
Collecting coefficients of terms which are homogeneous in $X$ determines a $\Sigma_k$-equivariant stable homotopy equivalence \[\Sigma^\infty_\ast\conf_k(S^m\times\RR^{n-m})\simeq\coprod_{i+j=k}\Bigl(\Sigma_k\underset{\Sigma_i}\times\Sigma^\infty_\ast\conf_i(\RR^n) \Bigr)\ot \Bigr(\Sigma_k\underset{\Sigma_j}\times\Sigma^\infty_\ast\bigl(\Sigma^{mj}_+\conf_j(\RR^{n-m})\bigr)\Bigl)\] where $\Sigma_k\underset{\Sigma_l}\times -$ is induction from $\Sigma_l$-spectra to $\Sigma_k$-spectra.

Now consider the general case for $\cV$, according to the hypothesis.
After Proposition~\ref{free}, both sides of the equivalence split as coproducts in homogeneous terms $V^{\ot k}$, so it suffices to show that the coefficients of these terms are degreewise equivalent. 
Inspecting, we thus seek an equivalence in $\cV$:
\[
\conf_k(S^m\times\RR^{n-m})\underset{\Sigma_k}\otimes V^{\ot k}
~\simeq~
\bigoplus_{i+j=k}\Bigl(\conf_i(\RR^n)\underset{\Sigma_i} \otimes V^{\ot i} \Bigr)\ot \Bigr(\conf_j(\RR^{n-m})\underset{\Sigma_j}\otimes (\Sigma^m V)^{\ot j}\Bigl)~.\]
This equivalence follows from the conclusion of the previous paragraph upon tensoring with $V^{\ot k}$ and taking balanced $\Sigma_k$-coinvariants.  

\end{proof}
\begin{remark} The equivalence in the above proposition can be upgraded to an equivalence of $(n-m)$-disk algebras if the righthand side is given a twisted algebra structure, using a natural action of $\free_{n-m}(\Sigma^m V)$ on $\free_n(V)$.
\end{remark}

The calculations to this point allow the following interesting description of the bar construction on a free $n$-disk algebra. Let $\cV$ be a symmetric monoidal $\oo$-category which is $\ot$-presentable.

\begin{prop}\label{bar} For any object $V\in\cV$ in a symmetric monoidal $\infty$-category which is $\ot$-presentable, there is an natural equivalence 
\[
{\sf Bar}\bigl(\free_n(V)\bigr)~\simeq~\free_{n-1}(\Sigma V)
\] 
between the bar construction on the free $n$-disk algebra on $V$ and the free $(n-1)$-disk algebra generated by the suspension of $V$.

\end{prop}

\begin{proof} 
Via Example~\ref{ex.assoc}, each augmented associative algebra $A\to \uno$ in $\cV$ determines a symmetric monoidal functor $A\colon \disk^{\partial, \sf or}_1\to \cV$. 
Applying $\ot$-excision in this simplest case of the collar-gluing $[-1,1)\underset{(-1,1)}\bigcup (-1,1]\cong [-1,1]$, we have that the bar construction ${\sf Bar}(A)$ is identifiable as the factorization homology over the closed 1-disk:
\[
{\sf Bar} (A) ~\simeq~\int_{\DD^1} A~.
\]
Proposition~\ref{free} gives the first and last of the following identifications 
\begin{eqnarray}
\nonumber
\underset{\DD^1\times\RR^{n-1}}\int \free_n(V) 
&
\simeq 
&
\coprod_{i\geq 0} \conf_i(\DD^1\times\RR^{n-1}, \partial \DD^1\times\RR^{n-1})\underset{\Sigma_i}\ot V^{\ot i} 
\\
\nonumber
&
\simeq 
&
\coprod_{i\geq 0} \Sigma^i\conf_i(\RR^{n-1})\underset{\Sigma_i}\ot V^{\ot i} 
\\
\nonumber
&
\simeq 
&
\coprod_{i\geq 0} \conf_i(\RR^{n-1})\underset{\Sigma_i}\ot (\Sigma V)^{\ot i} 
\\
\nonumber
&
\simeq 
&
\free_{n-1}(\Sigma V)~.
\end{eqnarray}
The second identification follows from the $\Sigma_i$-equivariant equivalence of spaces \[\conf_i(\DD^1\times M, \partial \DD^1\times M)\simeq \DD^i\times \conf_i(M)\big/\partial\DD^i\times \conf_i(M)\simeq \Sigma^i \conf_i(M)\] in the case $M=\RR^{n-1}$.
The third equivalence is a coproduct of a composite of two equivalences: $\Sigma^iX\ot V^{\ot i}\simeq X\ot \Sigma^i(V^{\ot i})\simeq X\ot (\Sigma V)^{\ot i}$.  
The first of these equivalences uses that tensoring with spaces preserves colimits among spaces -- an assertion which is direct from definitions.  
The second of these equivalences directly uses the assumption that the symmetric monoidal structure of $\cV$ distributes over colimits.    

\end{proof}

\begin{remark} In \cite{cotangent}, it was proved that ${\sf Bar}^n \free_n(V)\simeq \uno \oplus \Sigma^n V$, the free $0$-disk algebra on the $n$th suspension of $V$. This can now be seen as an application of Proposition \ref{bar} iterated $n$ times. This result is well-known in the case of $n=1$: the bar construction for the tensor algebra on $V$ is $\uno \oplus \Sigma V$. Our result is also entirely to be expected given the example of $n$-fold loop space, where for a connected pointed space $V$, we can calculate ${\sf Bar}\bigl( \free_n(V)\bigr)\simeq \sB\Omega^n\Sigma^nV\simeq \Omega^{n-1}\Sigma^nV\simeq \free_{n-1}(\Sigma V)$. As such, this result could have been proved longed ago, as it fits naturally into works such as \cite{may} and \cite{cohen}.
We note lastly that the limiting statement as $n$ increases gives the well-known equivalence ${\sf Bar}\bigl({\sf Sym}(V)\bigr)\simeq {\sf Sym}(\Sigma V)$.\end{remark}

This result has an important consequence for the relation between the $\oo$-categories of augmented $n$-disk algebras and augmented $(n-1)$-disk algebras:

\begin{theorem}\label{baradj} Let $\cV$ be a symmetric monoidal $\oo$-category which is $\ot$-presentable. There is an adjunction
\[{\sf Bar}: \Alg_{\disk_n^{\fr}}^{\sf aug}(\cV)\leftrightarrows \Alg_{\disk_{n-1}^{\fr}}^{\sf aug}(\cV): \Omega\]
where the functors are given by the bar construction and by a functor $\Omega$ which, on underlying objects of $\cV$, is the based loop functor. 
\end{theorem}
\begin{proof}
We first show that the bar construction defines a functor ${\sf Bar}\colon \Alg_{\disk_n^{\fr}}^{\sf aug}(\cV)\ra \Alg_{\disk_{n-1}^{\fr}}^{\sf aug}(\cV)$.
This is so in as much as ${\sf Bar}\simeq\int_{\DD^1\times\RR^{n-1}}$ is the object in $\cV$ underlying the augmented $\cE_{n-1}$-algebra $\disk^{\sf fr}_{n-1} \xra{\DD^1\times - } \mfld^{\partial,\sf fr}_n \xra{\int A} \cV$.
We will now argue that this functor carries colimit diagrams to colimit diagrams.

Proposition~\ref{bar} gives a commutative diagram among $\infty$-categories:
\[\xymatrix{
\Alg_{\disk^{\fr}_{n}}^{\sf aug}(\cV)\ar[rr]^{{\sf Bar}}&& \Alg_{\disk^{\fr}_{n-1}}^{\sf aug}(\cV)\\
\cV\ar[u]^{\free_n}\ar[rr]^\Sigma && \cV\ar[u]_{\free_{n-1}}.\\}\]
As a consequence, the functor ${\sf Bar}$ preserves coproducts of free $\cE_n$-algebras. 
We next argue that ${\sf Bar}$ preserves sifted colimits.
Using that the symmetric monoidal structure of $\cV$ distributes over colimits, it is enough to argue that factorization homology $\int_{M}\colon \Alg^{\sf aug}_{\disk^{\sf fr}_n}(\cV) \to \cV$ carries sifted colimit diagrams to colimit diagrams, for each framed $n$-manifold $M$ possibly with boundary.

So let $A\colon J\to \Alg_{\disk^{\sf fr}_n}^{\sf aug}(\cV)$ be a diagram of augmented $\cE_n$-algebras in $\cV$, indexed by a sifted $\infty$-category $J$.
The canonical arrow $\underset{j\in J}\colim~ {\sf Bar} (A_j) \longrightarrow {\sf Bar}( \underset{j\in J}\colim A_j)$ in $\cV$ is a composite
\[
\colim_{j\in J} \colim_{U\in \disk^{\partial,\sf fr}_{n/M}} A_j(U)~ \simeq ~\colim_{U\in \disk^{\partial, \sf fr}_{n/(M}} \colim_{j\in J} A_j(U) \longrightarrow \colim_{U\in \disk^{\partial, \sf fr}_{n/M}} (\colim_{j\in J} A_j)(U) 
\]
where the outer objects are in terms of the defining expression for factorization homology, the left equivalence is through commuting colimits, and the right arrow is a colimit of canonical arrows.  
Again using that the symmetric monoidal structure of $\cV$ distributes over colimits, 
each arrow $\underset{j\in J}\colim A_j(U) \to (\underset{j\in J}\colim A_j)(U)$ is an equivalence if and only if it is for $U$ connected.  
This is the case provided the forgetful functor ${\sf ev}_{\RR^n}\colon \Alg_n^{\sf aug}(\cV) \to \cV$ preserves sifted colimits.  
This assertion is Proposition~3.2.3.1 of~\cite{dag}.

Continuing, we conclude that ${\sf Bar}$ preserves all coproducts, since any coproduct is a geometric realization of coproducts of free algebras (this is a consequence of the $\infty$-categorical  Barr--Beck Theorem~4.7.4.5 of~\cite{dag}; see~\S4.7 thereof for a general discussion). 
Now, coproducts and geometric realizations generate all colimits, and we conclude that ${\sf Bar}$ is a colimit preserving functor from $n$-disk algebras to $(n-1)$-disk algebras.

To complete the proof, both of the $\infty$-categories in the adjunction are presentable (see Corollary 3.2.3.3 of~\cite{dag}).  
The adjoint functor theorem (Corollary~5.5.2.9 of~\cite{topos}) can thus be applied to conclude that ${\sf Bar}$ is a left adjoint. The diagram above is therefore a commutative diagram of left adjoints, and therefore their right adjoints commute. Consequently, ${\sf Bar}$ has a right adjoint which, at the level of objects of $\cV$, agrees with based loops $\Omega$, which is right adjoint to suspension $\Sigma$.  

\end{proof}

\begin{remark} We interpret Theorem~\ref{baradj} in terms of Koszul duality, after \cite{gk} and \cite{priddy}. Given the calculation of the \emph{Koszul dual} operad $\DD\cE_{n} \simeq \cE_n[-n]$, computed at the level of homology by Getzler and Jones \cite{getzlerjones} and computed in chain complexes by Fresse \cite{fressekoszul}, these functors should be equivalent to restriction and induction along the Koszul dual of the map $\cE_{n-1} \ra \cE_{n}$. However, Theorem \ref{baradj} is more general: it holds unstably (for instance, when $\cV$ is $\spaces$), whereas this operadic form of Koszul duality would require $\cV$ to be stable.
\end{remark}

\subsection{Factorization homology from Lie algebras}
We now discuss factorization homology of $n$-disk algebras coming from Lie algebras. Our results are closely analogous to those above about the factorization homology of $n$-disk algebras coming from topological spaces. For simplicity, we assume our Lie algebras are defined over a fixed field $k$ of characteristic zero.

As we proceed, we make use of the fact that Lie algebras in ${\sf Mod}_k$ admit totalizations, and therefore the $\infty$-category of such is cotensored over pointed spaces in a natural and standard way: $(X,\mathfrak{g})\mapsto {\mathfrak{g}}^X$. For $M$ an $n$-manifold, we notate $\Map_{\sf c}(M , \mathfrak{g}) := {\mathfrak{g}}^{M^+}$, where $M^+$ is the 1-point compactification. One can describe this as $\Mapc(M,\frak g)\simeq \sC_{\sf c}^\ast(M,\frak g)$, the compactly supported cochains of $M$ with coefficients in $\frak g$. See also \cite{owen} and \cite{kevinowen} for a discussion of the following.

\begin{prop}\label{env} For $\mathfrak{g}$ a Lie algebra over $k$, there is a natural equivalence of chain complexes over $k$,
\[\int_M \sC_\ast^{\Lie}\bigl(\Mapc(\RR^n,\mathfrak{g})\bigr)~ \simeq~ \sC^{\Lie}_\ast\bigl(\Mapc(M, \mathfrak{g})\bigr)~, \]
between the factorization homology of the Lie algebra chains of $\Omega^n {\mathfrak{g}}$ and the Lie algebra chains of the Lie algebra $\mathfrak{g}^{M^+}$.  

\end{prop}

\begin{proof} 
Lie algebra chains defines a functor between $\infty$-categories
$\sC^{\Lie}_\ast: \Alg_{\Lie}({\sf Mod}_k) \ra {\sf Mod}_k$.
This functor carries finite products of Lie algebras to finite tensor products of $k$-modules, which is to say that it is symmetric monoidal.
Furthermore, this functor preserves geometric realizations, so Lemma \ref{switch} applies to give a natural identification: $\int_M \sC_\ast^{\Lie}\bigl(\Mapc(\RR^n,\mathfrak{g})\bigr)\simeq  \sC_\ast^{\Lie}\bigl(\int_M\Mapc(\RR^n,\mathfrak{g})\bigr)$. The equivalence $\int_M\Mapc(\RR^n,\mathfrak{g})\simeq \Mapc(M,\mathfrak{g})$ now follows from the argument of nonabelian Poincar\'e duality (Corollary~\ref{non-abel}), only here we apply it in the usual abelian setting.\footnote{See \cite{af2} for a complete account, where nonabelian Poincar\'e duality is an instance of a version of Poincar\'e/Koszul duality for \emph{Cartesian-presentable $\infty$-categories}.}

\end{proof}

\begin{remark} 
The $n$-disk algebra $\sC_\ast^{\Lie}\bigl(\Mapc(\RR^n,\mathfrak{g})\bigr)$ has an interesting separate interpretation that we state here, and prove as a separate work. 
There is a forgetful functor from $\cE_n$-algebras in chain complexes over $k$ to Lie algebras over $k$ (see~\cite{cohen} for an account at the level of homology).  
The adjoint functor theorem (Corollary~5.5.2.9 of~\cite{topos}) applies to this functor, and so there is an adjunction
\[
\sU_n \colon \Alg_{\Lie}({\sf Mod}_k)~ \rightleftarrows~ \Alg_{\disk_n^{\sf fr}}({\sf Mod}_k)\colon {\sf fgt}~.
\]
In the case $n=1$, this left adjoint $\sU_1$ agrees with the familiar universal enveloping algebra functor.  
In general, there is an identification of $\disk_n^{\sf fr}$-algebras,
\[
\sU_n \mathfrak{g}~\simeq~  \sC_\ast^{\Lie}\bigl(\Mapc(\RR^n, \mathfrak{g})\bigr)~, 
\]
through which Proposition~\ref{env} can be reformulated as an equivalence of chain complexes over $k$:
\[
\int_M\sU_n\mathfrak{g}~\simeq~ \sC^{\Lie}_\ast\bigl(\Mapc(M,\mathfrak{g})\bigr)~. 
\]
\end{remark}

\end{document}